\documentclass{cmslatex}
\usepackage[paperwidth=7in, paperheight=10in, margin=.875in]{geometry}
\usepackage[backref,colorlinks,linkcolor=red,anchorcolor=green,citecolor=blue]{hyperref}
\usepackage{amsfonts,amssymb}
\usepackage{amsmath}
\usepackage{graphicx}
\usepackage{cite}
\usepackage{enumerate}
\allowdisplaybreaks[4]
\renewcommand{\theequation}{\arabic{section}.\arabic{equation}}

\def \u {\boldsymbol{u}}
\def \D {\mathcal{D}}
\def \v {\boldsymbol{v}}
\def \w {\boldsymbol{w}}
\def \g {\mathbf{g}}

\def \V {\boldsymbol{V}_{\!\!\!\!\!\!\!\!\sigma}}
\def \L {\boldsymbol{L}^2_{\sigma}}
\def \H {\mathbf{H}}

\def \d {\mathrm{d}}
\def \dt {\mathrm{d}t}
\def \dx {\mathrm{d}x}

\renewcommand{\theequation}{\arabic{section}.\arabic{equation}}

\begin{document}
\title{Global well-posedness for
a two-dimensional Navier-Stokes-Cahn-Hilliard-Boussinesq system
with singular potential
%\thanks{Received date, and accepted date (The correct dates will be entered by the editor).}
}

%For each author, make a block with the following macros:
\author{Lingxi Chen
\thanks{School of Mathematical Sciences, Fudan University, Handan Road 220, Shanghai 200433, People's Republic of China, (chenlingxi@gmail.com).}
}

\pagestyle{myheadings}
\markboth{Navier-Stokes-Cahn-Hilliard-Boussinesq System with Singular Potential}{L.-X. Chen}
\maketitle

\begin{abstract}
We study a general Navier-Stokes-Cahn-Hilliard-Boussinesq system 
that describes the motion of a mixture of two incompressible Newtonian fluids with thermo-induced Marangoni effects. The Cahn-Hilliard dynamics of the binary mixture is governed by aggregation/diffusion competition of the free energy with a physically-relevant logarithmic potential. The coupled system is studied in a bounded smooth domain $\Omega\subset \mathbb{R}^2$ and is supplemented with a no-slip condition for the fluid velocity, homogeneous Neumann boundary conditions for the order parameter and the chemical potential, homogeneous Dirichlet boundary condition for the relative temperature, and suitable initial conditions. For the
corresponding initial boundary value problem, we first prove the existence of global weak solutions and their continuous dependence with respect to the initial data. Under additional assumptions on the initial data, we prove the existence and uniqueness of a global strong solution and the validity of the strict separation property.
\end{abstract}


%%%%%%%%%%%%%%%%%%%%%%%%%%%%%%%%%%%%%%%%%%%%%
\section{Introduction}
\label{Introduction}
In this study, we consider the following Navier-Stokes-Cahn-Hilliard-Boussinesq system with temperature-dependent viscosity, thermal conductivity and surface tension in $\Omega \times (0,T)$:
\begin{align}
\u_t+\u \cdot \nabla \u 
- \nabla \cdot(2 \nu(\theta) \mathcal{D} \u) 
+ \nabla p &= 
- \nabla \cdot \sigma
+ ( \text{Ra}\theta - \text{Ga} ) g \mathbf{e}_2,
\label{NSCHM}\\
	\nabla \cdot \u &= 0,
\label{INCOMPRESS}\\
	\phi_t+\u \cdot \nabla \phi &= \Delta \mu,
\label{CH1}\\
	\mu &= -\Delta \phi + W^{\prime}(\phi),
\label{CH2}\\
	\theta_t+\u \cdot \nabla \theta-\nabla \cdot(\kappa(\theta) \nabla \theta) &= 0.
\label{BOUSSINESQ}
\end{align}
Here, $\Omega\subset \mathbb{R}^2$ is a bounded domain with smooth boundary and $T>0$ is the final time. 
The symmetric gradient is defined by $\D = \frac{1}{2} (\nabla + \nabla^t) $. This system is completed with the following initial and boundary conditions:
\begin{align}
	&\left.\u\right|_{t=0}=\u_0(x),\left.\quad \phi\right|_{t=0}=\phi_0(x),\left.\quad \theta\right|_{t=0}=\theta_0(x), \quad \text{in}\ \Omega ,
\label{initial} \\
	&\left.\u\right|_{\partial \Omega}=\mathbf{0},
	\left.\quad \partial_\mathbf{n} \phi \right|_{\partial \Omega}
	= \left. \partial_\mathbf{n} \mu \right|_{\partial \Omega}=0, \left.\quad \theta\right|_{\partial \Omega}=0,
	\quad \text{on}\ \partial \Omega \times (0, T),
\label{boundary}
\end{align}
where $\mathbf{n}$ is the unit outward normal vector on $\partial\Omega$, and $\partial_\mathbf{n}$ denotes the outer normal derivative on $\partial\Omega$. The unknown variables of the system are denoted by $(\u, p, \phi, \mu, \theta)$. The vectorial function $\u$ is the (volume) averaged velocity of the binary fluids, admitting the incompressible Navier-Stokes equations, where $p$ is the pressure. The scalar function $\phi$ refers to the so-called order parameter (or phase-field), together with chemical potential $\mu$ consisting the Cahn-Hilliard equation. The scalar function $\theta$ represents the relative temperature and satisfies the Boussinesq system.

The coupled system (\ref{NSCHM})-(\ref{BOUSSINESQ}) describes the dynamics of a two-phase flow of incompressible viscous Newtonian fluids, where surface tension inhomogeneity on the interface caused by the temperature difference is considered. This phenomenon is referred to as the (thermal) Marangoni effect in the literature \cite{Marangoni,EnVarA1}. The system adopts the diffuse interface framework and contains a (convective) Cahn-Hilliard equation as a subsystem, where sharp interface (see e.g., \cite{sharp1,sharp2}) of macroscopically immiscible fluids are substituted by a capillary layer. Over this transition layer, physical quantities have steep but smooth changes, which is convenient for both mathematical and numerical analysis (see e.g., \cite{phasefield1,phasefield2,phasefield3,phasefield4} for advantages of the diffuse interface models). The model under consideration was first derived in \cite{EnVarA1,EnVarA2} by the energetic variational approach. For a more complicated model with temperature-dependent singular potential, we refer to the recent work \cite{Liuextend}.

In the above equations, the temperature-dependent coefficients  $\nu(\theta)$ and $\kappa(\theta)$ denote the fluid viscosity  and the thermal conductivity, respectively.
In \eqref{NSCHM}, $\mathbf{e}_2=(0,1)^\mathrm{T}$ is the unit vector along $y$-axis and the term $( \text{Ra}\theta - \text{Ga} ) g \mathbf{e}_2$ stands for the Rayleigh-Galileo approximation of the buoyancy force (see e.g., \cite{Gills,Joseph}).
The Cauchy stress tensor $\sigma$ is given by
\begin{align}\label{Cauchy_stress_tensor}
\sigma =  \lambda(\theta) ( \nabla \phi \otimes \nabla \phi ) +  \lambda(\theta)
\left( \frac{1}{2}|\nabla \phi|^2 + W(\phi)\right)  \mathbb{I}_2,
\end{align}
where $\mathbb{I}_2$ is the two-dimensional unit matrix.
In \eqref{Cauchy_stress_tensor}, the function $\lambda(\theta)$ denotes the temperature-dependent surface tension approximated by the $E\ddot{o}tv\ddot{o}s$ rule $$\lambda(\theta) = \lambda_0(a - b \theta),$$ with $a$ and $b$ being two constants (see \cite{Eotvos}). Besides, $W(\phi)$ is the physically-relevant logarithmic potential \cite{Cahn1958} (also known as the Flory-Huggins potential) such that
\begin{align} \label{Wphi}
W(\phi) = \frac{A}{2} \big[ (1 + \phi) \ln(1+\phi) + (1 - \phi) \ln(1-\phi)     \big]
- \frac{B}{2} \phi^2,\quad \phi\in (-1,1),
\end{align}
where $A$ and $B$ are constants satisfying $0<A<B$. It's easy to check that  $W^{\prime\prime}(\phi) \geqslant -\alpha$, for some $\alpha > 0$ (see e.g. \cite{MBoussinesq}). In this paper, we denote the convex part by
\begin{align}\label{F}
F(\phi) := \dfrac{A}{2} \big[ (1 + \phi) \ln(1+\phi) + (1 - \phi) \ln(1-\phi)    \big].
\end{align}

In the literature, there is little investigation for the well-posedness of the initial boundary value problem (\ref{NSCHM})-(\ref{boundary}). The inviscid case with a regular potential has been analyzed
in \cite{Zhao_regularity,Zhao_longtime1,Zhao_longtime2} in a two-dimensional bounded domain.
In \cite{Zhao_regularity}, the author proved the global existence of classical solutions with smooth initial data. Moreover, in \cite{Zhao_longtime1,Zhao_longtime2}, he considered the long time behavior with the same assumption for the initial data, and proved the exponential convergence rate to the equilibrium. The former considered the case where the mobility is constant, and the latter extended this to the case where the mobility depends on the order parameter. In a recent work \cite{MBoussinesq}, the authors considered a similar system with a singular potential, constant surface tension and thermal conductivity. They proved the existence of strong solutions by a standard Faedo-Galerkin method, which was based on a suitable regularization of the singular potential. 
In \cite{HWUXX2018,HW2017},  the authors considered an alternative system, where the Cahn-Hilliard equation with a singular potential was replaced by the Allen-Cahn equation with a regular potential. In this manner, possible difficulties related to the singular potential can be avoided, while the phase function still satisfies the maximum principle. In \cite{HWUXX2018}, the authors considered the case where the thermal conductivity $\kappa$ is a constant. In \cite{HW2017}, the results were extended to the case with temperature-dependent thermal conductivity. With suitable initial and boundary conditions, both works proved the existence of global weak and strong solutions in two dimensions with a suitably small initial temperature. The bound of the temperature ensures the system to be a dissipative one. Under the same assumptions on the initial temperature, the former proved that every global weak solution converges to a single equilibrium as time goes to infinity. In the situation with a smaller initial temperature, the latter derived uniform-in-time estimates for weak and strong solutions in corresponding spaces.

In this study, we consider a general diffuse interface model for incompressible two-phase flows with thermal Marangoni effects, where the viscosity, surface tension and thermal conductivity are all temperature-dependent.
Our aim is to establish well-posedness for the initial boundary value problem (\ref{NSCHM})-(\ref{boundary}) in a two-dimensional setting, without any smallness assumption on the initial temperature.
The results are listed as follows:
\begin{enumerate}[(1)]
\item(Theorem \ref{2dweak-exist}) existence of a global weak solution   and uniform-in-time estimates;
\item(Theorem \ref{2dweak-unique}) continuous dependence with respect to the initial data and uniqueness of the weak solution;
\item(Theorem \ref{2dstrong}) existence and uniqueness of the global strong solution and the validity of strict separation property for the phase function $\phi$.
\end{enumerate}	
To achieve our goal, several difficulties due to the highly nonlinear structure of the system and the singular potential have to be overcome.
For instance, we shall take advantage of a novel interpolation inequality (\ref{GiorginiHolder}). Some techniques for handling the singular potential originate from \cite{Giorgini2019,phibound}. Besides, for the continuous dependence estimate, we apply a useful estimate inspired by \cite{Giorgini}.

It is worth mentioning that we are able to prove the existence of global weak and strong solutions without assuming that the initial temperature should be properly small. The proof relies on a new attempt of the so-called semi-Galerkin scheme, in which only part of the unknown variables are approximated (see e.g., \cite{Lin1995,HW2017,HWUXX2018,HeWeak,HeStrong,GiorginiGalerkin}).
Comparing with the previous results in \cite{HWUXX2018,HW2017}, where the Allen-Cahn equation was considered, we observe that the Cahn-Hilliard equation can improve the dissipation of the system. More precisely, the fourth order Cahn-Hilliard equation provides us a higher order dissipative term involving $\|\nabla \mu\|_{L^2}$ (see (\ref{nabla phi m})), while the second order Allen-Cahn equation only presents a lower order dissipation term involving $\|\mu\|_{L^2}$ (see \cite{HWUXX2018,HW2017}). The latter is not enough to control the change of energy when the temperature $\theta$ is large. Next, with the novel interpolation inequality (\ref{GiorginiHolder}), we can derive uniform-in-time estimates for the solutions and remove the previous requirement on the smallness of the initial temperature as in \cite{HW2017}. Uniqueness of the weak solution is in virtue of a recent $L^4$-estimate for the pressure in the two-dimensional Stokes problem (see \cite{Giorgini}). Taking advantage of this, we can handle the most difficult term $I_8$ in the subsequent proof (see Section \ref{Globalweak}). We also note that although the Cahn-Hilliard is a fourth order parabolic equation, which in general does not enjoy a maximum principle, the singular potential still yields an $L^\infty$-estimate for $\phi$. This together with the maximum principle for $\theta$ (see (\ref{semimaximum})) plays a crucial role to control highly nonlinear terms like $\lambda(\theta) (\nabla \phi \otimes \nabla \phi)$.

\textit{Plan of the paper.} In Section \ref{Preliminaries}, we first introduce notations, necessary mathematical tools and basic assumptions. Then we present the definition of weak and strong solutions, and state the main results. Section \ref{Globalweak-exist} focuses on proving the existence of a global weak solution and some uniform-in-time estimates. In Section \ref{Global Weak Solutions-Uniqueness}, we prove the continuous dependence with respect to the initial data. Section \ref{Global Strong Solutions} is devoted to the existence and uniqueness of the global strong solution and the strict separation property. In the Appendix, we provide some details of the semi-Galerkin scheme, properties of an elliptic problem with singular nonlinearity, and the proof of Lemma \ref{Giorginiinequality}.

\section{Preliminaries} \label{Preliminaries}
\subsection{Notations} \label{Pre-Notations}
For any Banach space $X$, we denote by $X^{\prime}$ its dual space, and by $\langle \cdot,\cdot \rangle_X$ the usual dual product. The boldfaced letter $\boldsymbol{X}$ denotes the space of vector-valued functions with every component belonging to $X$.
Let $\Omega \subset \mathbb{R}^2$ be a bounded domain with smooth boundary. We write $C(\overline{\Omega})$, $C^{\alpha}({\overline{\Omega}})$ for the spaces of continuous functions and $\alpha$-H\"{o}lder continuous functions defined in the closure of $\Omega$, respectively.
Besides, we denote by $C_0^{\infty}(\Omega)$ the space of infinitely differentiable functions with compact support in $\Omega$. For $1\leqslant p\leqslant \infty$, $L^p(\Omega)$ denotes the Lebesgue space of real measurable $p$-th power integrable/essentially bounded functions over $\Omega$, endowed with the norm $\| \cdot \|_{L^p}$. In particular, when $p=2$, $L^2(\Omega)$ becomes a Hilbert space with inner product denoted by $(\cdot,\cdot)$, and we write the norm $\| \cdot \|$, omitting the subscript. For $f \in L^p(\Omega)$, $\overline{f}$ stands for the integral mean value of $f$ over $\Omega$ with the notation $\overline{f} =  |\Omega|^{-1} \int_{\Omega} f(x) \,\mathrm{d} x$.
For $m \in \mathbb{N}$, $1\leqslant p\leqslant \infty$, we denote by $W^{m,p}(\Omega)$ the Sobolev spaces of real measurable functions with weak derivatives in $L^p(\Omega)$ of orders up to $m$ with usual equivalent norms $\|\cdot\|_{W^{m,p}}$. $W_0^{m,p}(\Omega)$ denotes the closure of $C_0^{\infty}(\Omega)$ in $W^{m,p}(\Omega)$.   When $p=2$, we use the notation $H^m(\Omega) = W^{m,p}(\Omega)$. Without ambiguity, we usually omit the domain and just write $L^p$, $W^{m,p}$ and $H^m$. For two matrices $U,V\in \mathbb{R}^{2\times2}$, we denote $U:V=\mathrm{tr}(U^\mathrm{T}V) = \sum \limits_{i,j} U_{ij}V_{ij}$, where $U_{ij}, V_{ij}$ represent the $(i,j)$-th entry of $U$ and $V$, respectively.

For convenience, we introduce the linear subspaces
$$V_0 = \{v \in H^1\ :\ \overline{v} = 0	\}, \qquad
V_0^{\prime} = \{g \in (H^1)^{\prime}\ :\  \langle g,1 \rangle_{H^1} = 0	\}.$$
Consider the Neumann problem
\begin{align} \label{Pre-Nuemannproblem}
\begin{cases}
-\Delta u = g, \quad &\text{in}~\Omega, \\
 \partial_\textbf{n} u = 0, \quad &\text{on}~\partial \Omega.
\end{cases}
\end{align}
We define the bounded linear operator $A_0 \in \mathcal{L}(V_0,V_0^{\prime}):$
\begin{align} \label{Pre-A0}
\langle A_0 u , v \rangle = \int_{\Omega} \nabla u(x) \cdot \nabla v(x) \,\mathrm{d}x, \quad \text{for}~ u, v \in V_0.
\end{align}
It follows from the Poincar\'e-Wirtinger inequality and the Lax-Milgram theorem that $A_0$ is a linear isomorphism from $V_0$ to $V_0^{\prime}$. Moreover, for $g \in V_0^{\prime},$ it is easy to check that $\| \nabla A_0^{-1} g \|$ is a norm on $V_0^{\prime}$ equivalent to the normal functional norm, so we also write this norm as $\|\cdot\|_{V_0^{\prime}}$. Given $g \in H^1(0,T;V_0^{\prime})$, we have $A_0^{-1} g(t) \in L^2(0,T;V_0)$, and the chain rule holds according to the Lions-Magenes theorem (see \cite{LionsMagenes}):
\begin{align} \label{Pre-LM1}
\dfrac{1}{2} \dfrac{\d}{\dt} \|g(t)\|_{V_0^{\prime}}^2 = \langle g_t(t), A_0^{-1} g(t) \rangle_{H^1},\quad \text{a.e.} ~ t \in (0,T).
\end{align}

Next, let us consider the Stokes problem with homogeneous Dirichlet boundary conditions:
\begin{align} \label{Pre-Stokes}
	\begin{cases}
		-\Delta \u + \nabla p = \boldsymbol{g}, &\text{in}~\Omega, \\
		\nabla \cdot \u = 0, &\text{in}~\Omega, \\
		\u = 0, &\text{on}~\partial \Omega.
	\end{cases}
\end{align}
We introduce the space $C_{0,\sigma}^{\infty}(\Omega)$ of solenoidal infinitely-differentiable functions with compact support in $\Omega$, and use the symbols $\L$, $\V$ to denote the closure of $C_{0,\sigma}^{\infty}$ in $\boldsymbol{L}^2$ and $\boldsymbol{H}^1$, respectively. The Stokes operator is the bounded linear operator $S \in \mathcal{L}(\V,\V^{\prime})$ such that
\begin{align} \label{Pre-StokesS}
\langle S \u, \v \rangle_{\V} = (\nabla \u, \nabla \v),  \quad \text{for}~ \u, \v \in \V.
\end{align}
Thanks to the Poincar\'e inequality and the Lax-Milgram theorem, $S$ is a linear isomorphism from $\V$ to $\V^{\prime}$. Furthermore, for $\boldsymbol{g} \in \V^{\prime}$, $\| \nabla S^{-1} \boldsymbol{g}\|$ is an equivalent norm on $\V^{\prime}$ to the natural one. In this paper, we write this norm as $\|\cdot\|_{\V^{\prime}}$. Assuming $\boldsymbol{g} \in \boldsymbol{H}^{-1}(\Omega)$, the pressure $\nabla p$  occurs due to the well-known de Rham's theorem (see e.g., \cite{Sohr,Gilles}). It is proved that (see \cite{Temam})
\begin{align} \label{Pre-Stokespressure}
\| p \| \leqslant C\|\boldsymbol{g}\|_{\boldsymbol{H}^{-1}}.
\end{align}
For each $\boldsymbol{g} \in H^1(0,T; \V^{\prime})$, we have $S^{-1} \boldsymbol{g} \in L^2(0,T;\V)$. Due to the Lions-Magenes theorem, it follows that
\begin{align} \label{Pre-LM2}
\frac{1}{2} \frac{\mathrm{d}}{\mathrm{d} t} \| \boldsymbol{g}(t) \|_{\V^{\prime}}^2 = \langle \boldsymbol{g}_t(t), S^{-1} \boldsymbol{g}(t) \rangle_{\V}, \quad \text{a.e.}~ t \in(0,T).
\end{align}

\subsection{Useful inequalities} \label{Pre-Inequality}
The following interpolation and elliptic estimates will be used in this paper:
\begin{align}
 \|u\| \leqslant \|u\|_{V_0^{\prime}}^{\frac{1}{2}} \|\nabla u\|^{\frac{1}{2}}, \quad &\text{for}~ u \in V_0, ~ V_0 \hookrightarrow V_0^{\prime} ~ \text{is the canonical injection}.
\label{Pre-V0interpolation} \\
\| \nabla A_0^{-1} g \|_{H^k} \leqslant& C\|g\|_{H^{k-1}},  \quad \text{for}~ g \in H^{k-1} \cap V_0, ~ k \in \mathbb{N}.
\label{Pre-V0elliptic}
\end{align}
When the spatial dimension is two, according to the well-posedness and regularity theory of the Stokes problem \eqref{Pre-Stokes}, for every $\boldsymbol{g} \in \boldsymbol{L}_\sigma^2$, there exists a unique pair $(\u,p) \in (\boldsymbol{H}^2 \cap \V) \times H^1$ with $\overline{p} = 0$, such that (see e.g., \cite{Gilles})
\begin{align} \label{Pre-Stokesestimate}
\| \u \|_{\boldsymbol{H}^2} + \|p\|_{H^1} \leqslant C\|\boldsymbol{g}\|.
\end{align}
Moreover, the following estimate on the $L^4$-norm of the pressure holds (see \cite[Lemma 3.1]{Giorgini}):
\begin{align}\label{pressure}
\|p\|_{L^4} \leq C\left\|\nabla S^{-1} \boldsymbol{g}\right\|^{\frac{1}{2}}  \|\boldsymbol{g}\|^{\frac{1}{2}},\quad \forall\,\boldsymbol{g} \in \boldsymbol{L}_\sigma^2.
\end{align}

We also recall the Korn's inequality
\begin{align}\label{Korn}
\| \nabla \u \| \leqslant \sqrt{2} \| \D \u \| \leqslant \sqrt{2} \| \nabla \u \|,   \quad \text{for}~ \u \in \V,
\end{align} 
where $\D \u = \frac{1}{2} ( \nabla \u + (\nabla \u)^t )$.

We report the Gagliardo-Nirenberg inequality that will be frequently used later (see e.g., \cite{Nirenberg}):
\medskip

\begin{lemma}[Gagliardo-Nirenberg Inequality] \label{Gagliardo-Nirenberg Inequality}
  Let $\Omega \subset \mathbb{R}^n$ be a bounded domain with smooth boundary. Let $j, m \in \mathbb{Z}$, $p,q,r \in \mathbb{R}$  satisfy  $0 \leqslant j<m$, $1 \leqslant q, r \leqslant \infty$, $\frac{j}{m} \leqslant a \leqslant 1$ (if $1<r<\infty$ and $m-j-\frac{n}{r}$ is a nonnegative integer, then  $a \neq 1$)  such that
\begin{align} \label{Inequality-dimension}
\frac{1}{p}-\frac{j}{n}=a\left(\frac{1}{r}-\frac{m}{n}\right)+(1-a) \frac{1}{q}.
\end{align}	
Then there are two positive constants $C_1, C_2$ depending only on $\Omega$, such that for any $u \in W^{m, r}(\Omega) \cap L^q(\Omega)$, the following inequality holds:
\begin{align} \label{Inequality-GagliardoNirenberg}
\|D^j u\|_{L^p} \leqslant C_1\|D^m u\|_{L^r}^a \|u\|_{L^q}^{1-a} +
C_2\|u\|_{L^q}.
\end{align}
In particular, for any $u \in W_0^{m, r}(\Omega) \cap L^q(\Omega)$, the constant $C_2$ can be taken as zero.
\end{lemma}
\medskip

In particular, we introduce a useful interpolation inequality, which states that the $W^{1,4}$-norm can be controlled by a H\"{o}lder-norm together with the $H^2$-norm (see \cite{GiorginiHolder}). Its proof can be found in the Appendix D.
\medskip
\begin{lemma}\label{Giorginiinequality}
Let $\Omega \subset \mathbb{R}^n$, $n=2,3$, be a bounded domain with smooth boundary. Given $\gamma \in (0,1)$, there exists $\xi \in (\frac{1}{2},1)$ such that for any $u \in H^2(\Omega) \cap C^{\gamma}(\overline{\Omega})$, the following inequality holds:
\begin{align} \label{GiorginiHolder}
\| u \|_{W^{1,4}} \leqslant C \| u \|_{C^{\gamma}(\overline{\Omega})}^{\xi} \| u \|_{H^2}^{1 - \xi}.
\end{align} 
\end{lemma}

Finally, we recall some \emph{a priori} estimates for the phase function $\phi$ (see \cite{Giorgini2019,inequality}).
\medskip

\begin{lemma}[$H^2$-estimates for $\phi$] \label{phipriori}
 Let $\Omega\subset \mathbb{R}^2$ be a bounded domain  with smooth boundary. Assume that $\phi$, $\mu$ are smooth functions satisfying the equation (\ref{CH2}) with homogeneous Neumann boundary condition $\partial_\mathbf{n}\phi=0$. Then we have
\begin{align} \label{Inequality-phi1}
	&\| \phi \|_{H^2}^2 \leqslant C \| \phi\|^2 + C (\nabla \mu, \nabla \phi)
	\leqslant C \| \phi\|^2 + C \|\nabla \mu\| \| \nabla \phi\|,
\\
\label{Inequality-phi2}
	&\| \phi \|_{H^2}^2 \leqslant C \| \phi\|^2 + \|\nabla \mu\|^2.
\end{align}
\end{lemma}
\begin{proof}
Multiplying (\ref{CH2}) by  $- \Delta \phi$, and integrating over $\Omega$, we get
\begin{align} \label{Proof-phi1}
\| \Delta \phi\|^2 = -W^{\prime \prime} (\phi) \| \nabla \phi\|^2 + (\nabla \mu, \nabla \phi)
\leqslant \alpha \| \nabla \phi\|^2 + (\nabla \mu, \nabla \phi).
\end{align}
Thanks to the standard elliptic estimates for the Neumann problem, it follows that
\begin{align}  \label{Proof-phi2}
\| \phi\|_{H^2}^2 \leqslant C \left( \alpha \| \nabla \phi\|^2 + (\nabla \mu, \nabla \phi) + \|\phi\|^2\right).
\end{align}
An interpolation for $\| \nabla \phi \|$ together with H\"older's and Young's inequalities completes the proof.
\end{proof}

\subsection{Main Results} \label{Pre-Results}
For the sake of simplicity (and also without loss of generality, see e.g. \cite{HW2017} for detailed modification of these coefficients), in this paper, we make the following assumption on the coefficients: \smallskip
\begin{itemize}
\item Assume $\nu, \kappa$ are second-order differentiable functions with positive lower bounds $\underline{\nu}$, $\underline{\kappa}$ and upper bounds $\overline{\nu}$, $\overline{\kappa}$, respectively. Moreover, we assume $\nu^{\prime},\nu^{\prime\prime},\kappa^{\prime},\kappa^{\prime\prime}$ are bounded.
\end{itemize}
\medskip

Next, we introduce the notion of weak and strong solutions.
\medskip

\begin{definition}[Weak Solutions]
Let $T \in(0,+\infty)$. Suppose that the initial data satisfy $\u_0 \in \L(\Omega)$, $\phi_0 \in H^1(\Omega)$ with  $\left\|\phi_0\right\|_{L^{\infty}} \leqslant 1$, $\left|\overline{\phi_0}\right|<1$, and $\theta_0 \in L^{\infty}(\Omega) \cap H_0^1(\Omega)$. We call $(\u, \phi, \mu, \theta)$ a weak solution to problem (\ref{NSCHM})-(\ref{boundary}) on $[0, T]$, if the following conditions are satisfied:
\begin{align*}
&\u \in L^{\infty}\left(0, T ; \boldsymbol{L}_{\sigma}^2(\Omega)\right) \cap L^2\left(0, T ; \V(\Omega)\right) \cap H^{1}\left(0, T ; \V^{\prime}(\Omega)\right), \\
&\phi \in L^{\infty}\left(0, T ; H^1(\Omega)\right) \cap L^4\left(0, T ; H^2(\Omega)\right) \cap L^2\left(0, T ; W^{2, p}(\Omega)\right) \cap H^1\left(0, T ; (H^1(\Omega))^{\prime}\right), \\
&\mu \in L^2\left(0, T ; H^1(\Omega)\right), \\
&\theta \in L^{\infty}\left(0, T ; L^{\infty}(\Omega) \cap H_0^1(\Omega) \right) \cap L^2\left(0, T ; H^2(\Omega)\right) \cap H^1\left(0, T ; L^2(\Omega)\right), \\
&\phi \in L^{\infty}(\Omega \times(0, T)), \quad \text { with } |\phi(x, t)|<1 \  \text { a.e. in } \Omega \times(0, T),
\end{align*}
where $p \geqslant 2$ is arbitrary. Then,
\begin{align}
\left\langle\u_t, \v\right\rangle_{V_{\sigma}}+(\u \cdot \nabla \u, \v)+( 2 \nu(\theta) \D \u, \nabla \v) &=\int_{\Omega}[\lambda(\theta) \nabla \phi \otimes \nabla \phi]: \nabla \v\, \mathrm{d} x +
\int_{\Omega} \theta \g \cdot \v\, \mathrm{d} x,
\label{weaku} \\
\left\langle\phi_t , \xi\right\rangle_{H^1}+(\u \cdot \nabla \phi, \xi) + (\nabla \mu, \nabla \xi)&=0,
\label{weakphi}
\end{align}
for all $\v \in \V$, $\xi \in H^1(\Omega)$ and almost all $t \in (0,T)$, besides,
\begin{align}
&\mu = - \Delta \phi + W^{\prime}(\phi),\\
&\theta_t+\u \cdot \nabla \theta-\nabla \cdot(\kappa(\theta) \nabla \theta) =0,
\label{weaktheta}
\end{align}
 almost everywhere in $\Omega\times (0,T)$. In \eqref{weaku}, $ \g$ denotes the abbreviation of $\mathrm{Ra} g \mathbf{e}_2$.
 Moreover, we have $ \partial_\mathbf{n} \phi   =  \theta = 0$ almost everywhere on $\partial\Omega\times (0,T)$, and the initial conditions (\ref{initial}) are satisfied almost everywhere in $\Omega$.
\end{definition}
\medskip

\begin{definition}[Strong Solutions]
Let $T \in (0, + \infty)$. Suppose that the initial data satisfy $\u_0 \in \V$, $\phi_0 \in H^2(\Omega)$ with $\partial_\mathbf{n}\phi_0=0$ on $\partial\Omega$, $\left\|\phi_0\right\|_{L^{\infty}} \leqslant 1$, $\left|\overline{\phi_0}\right|<1$, $\mu_0 := -\Delta \phi_0 + W^{\prime}(\phi_0) \in H^1(\Omega)$ and $\theta_0 \in H^2(\Omega) \cap H_0^1(\Omega)$. We call $(\u,\phi,\mu,\theta)$ a strong solution to problem (\ref{NSCHM})-(\ref{boundary}) on $[0, T]$, if
\begin{align*}
&\u \in L^{\infty}\left(0, T ; \V(\Omega)\right)
\cap L^2\left(0, T ; \boldsymbol{H}^2(\Omega)\right)
\cap H^1\left(0, T ; \boldsymbol{L}_\sigma^{2}(\Omega)\right),  \\
&\phi \in L^{\infty}(0,T; H^3(\Omega)) \cap
L^2(0,T; H^4(\Omega)) \cap
H^1(0,T; H^1(\Omega)),  \\
&\mu \in L^{\infty}(0,T; H^1(\Omega)) \cap
L^2(0,T; H^3(\Omega)) \cap
H^1(0,T; (H^1(\Omega))^{\prime}), \\
&\theta \in L^{\infty}(0,T; H^2(\Omega)) \cap
L^2(0,T; H^3(\Omega)) \cap
W^{1,\infty} (0,T; L^2(\Omega)) \cap
H^1(0,T; H_0^1(\Omega)), \\
&\phi \in L^{\infty}(\Omega \times(0, T)), \quad \text { with }\ |\phi(x, t)|<1\ \text { a.e. in } \Omega \times(0, T),
\end{align*}
and $(\u,\phi,\mu,\theta)$ satisfy equations (\ref{NSCHM})-(\ref{BOUSSINESQ}) almost everywhere in $\Omega \times (0,T)$ with initial and boundary conditions (\ref{initial}), (\ref{boundary}).	
\end{definition}
\medskip

Now we are in a position to state the main results of the paper.
\medskip

\begin{thm}[Existence of a Global Weak Solution] \label{2dweak-exist}
Suppose that $\u_0 \in \L(\Omega)$, $\phi_0 \in H^{1}(\Omega)$ with $\left\|\phi_0\right\|_{L^{\infty}} \leqslant 1$,
$\left|\overline{\phi_0}\right|<1$, and $\theta_0 \in L^{\infty}(\Omega) \cap H_0^{1}(\Omega)$.
Problem (\ref{NSCHM})-(\ref{boundary})
admits a unique global weak solution on $[0,\infty)$. If in addition, $\theta_0 \in C^{\gamma}(\overline{\Omega})$ for some $\gamma \in (0,1)$, then we have
\begin{align}
&\| \u \|_{L^{\infty}(0, \infty; \L(\Omega))} + \| \phi\|_{L^{\infty}(0, \infty; H^1(\Omega))} +\| \theta \|_{L^{\infty}(0, \infty; C^{\beta}(\overline{\Omega}) \cap H^1(\Omega))} + \sup\limits_{\tau \geqslant 0} \| \u\|_{L^2 ( \tau, \tau + 1 ; \V(\Omega))} \notag \\
&\quad + \sup\limits_{\tau \geqslant 0} \| \u\|_{L^4 ( \tau, \tau + 1 ; \boldsymbol{L}^4(\Omega))} +
\sup\limits_{\tau \geqslant 0} \| \mu \|_{L^2 ( \tau, \tau + 1 ; H^1(\Omega))} +
\sup\limits_{\tau \geqslant 0} \| \theta \|_{L^2( \tau, \tau + 1 ; H^2(\Omega))}
\leqslant C,
\end{align}
for some $ \beta \in (0, \gamma]$. Here, the positive constant $C$ depends on the initial data, $\Omega$ and parameters of the system. 
\end{thm}
\medskip

\begin{remark}
The conclusion of Theorem \ref{2dweak-exist} only holds in two dimensions. In the three-dimensional case, the approximate solution $\theta^m$ in the semi-Galerkin scheme (\ref{semiu})-(\ref{semiboundary}) is not continuous with respect to the velocity field $\v^m$ (see Appendix A for details). Nevertheless, if we further assume that $\theta^m(\cdot, 0) \in H^2(\Omega)$, the approximate system (\ref{semiu})-(\ref{semiboundary}) will admit a local solution with suitable regularity. The three-dimensional case will be analyzed in a forthcoming paper.
\end{remark}
\medskip

Thanks to the $L^4$-estimate (\ref{pressure}) for the pressure in the Stokes problem, we can prove the following continuous dependence result, which directly yields uniqueness of the weak solution. 
\medskip

\begin{thm}[Continuous Dependence] \label{2dweak-unique}
Let $(\u_1,\phi_1,\mu_1,\theta_1)$, $(\u_2,\phi_2,\mu_2,\theta_2)$ be two weak solutions to problem (\ref{NSCHM})-(\ref{boundary}) defined on $[0,T]$ subject to initial data $(\u_{01},\phi_{01},\theta_{01})$, $(\u_{02},\phi_{02},\theta_{02})$, respectively. Write
\begin{align*}
(\u, \phi, \theta) &= (\u_{1} - \u_{2},\phi_{1} - \phi_{2}, \theta_{1} - \theta_{2}), \\
(\u_0,\phi_0,\theta_0)   &= (\u_{01} - \u_{02},\phi_{01} - \phi_{02},\theta_{01} - \theta_{02}).
\end{align*}
Then we have
\begin{align} \label{continuousdependence}
\|\u(t)\|_{\V^{\prime}}^2 + \|\phi(t) - \overline{\phi}(t)\|_{V_0^{\prime}}^2 + \|\theta(t)\|^2 \leqslant (\|\u_0\|_{\V^{\prime}}^2 + \|\phi_0 - \overline{\phi_0}\|_{V_0^{\prime}}^2 + \|\theta_0\|^2) e^{C_T},
\end{align}
for all $t\in[0,T]$.
Here, $C_T > 0$ is a constant depending on norms of the initial data, $T$, $\Omega$ and parameters of the system.
\end{thm}
\medskip

Our last result concerns the existence and uniqueness of global strong solutions.
\medskip

\begin{thm}[Existence of a Global Strong Solution] \label{2dstrong}
Suppose that $\u_0 \in \V(\Omega)$, $\phi_0 \in H^{2}(\Omega)$ with $\partial_\mathbf{n}\phi_0=0$ on $\partial\Omega$, $\left\|\phi_0\right\|_{L^{\infty}} \leqslant 1$,
$\left|\overline{\phi_0}\right|<1$, $\mu_0 := -\Delta \phi_0 + W^{\prime}(\phi_0) \in H^1(\Omega)$ and $\theta_0 \in  H^2( \Omega) \cap H_0^{1}(\Omega)$. Then problem (\ref{NSCHM})-(\ref{boundary}) admits a unique global strong solution on $[0,\infty)$. Moreover, there exists some $\delta \in (0,1)$ such that
\begin{align} \label{separation}
\|\phi(t)\|_{C(\overline{\Omega})} \leqslant 1-\delta, \quad \forall\, t \geqslant 0.
\end{align}
\end{thm}

\medskip

\begin{remark}
The argument in Section 4.2 yields the following higher order uniform-in-time estimates:
\begin{align*}
&\| \u \|_{L^{\infty}(0 , \infty ; \V)} + \sup\limits_{\tau \geqslant 0}	\| \u \|_{L^2(\tau , \tau + 1 ; \boldsymbol{H}^2)}
+ \sup\limits_{\tau \geqslant 0} \| \u_t \|_{L^2(\tau , \tau + 1 ; \L)}  \\
&\quad + \| \mu \|_{L^{\infty}(0 , \infty ; H^1)} + \sup\limits_{\tau \geqslant 0}	\| \mu \|_{L^2(\tau , \tau + 1 ; H^3)}
+ \sup\limits_{\tau \geqslant 0}	\| \mu_t \|_{L^2(\tau , \tau + 1 ; (H^1)^{\prime})}  \\
&\quad + \| \phi \|_{L^{\infty}(0 , \infty ; H^3)} + \sup\limits_{\tau \geqslant 0}	\| \phi \|_{L^2(\tau , \tau + 1 ; H^4)}
+ \sup\limits_{\tau \geqslant 0} \| \nabla \phi_t \|_{L^2(\tau , \tau + 1 ; \boldsymbol{L}^2)}  \\
&\quad + \| \theta \|_{L^{\infty}(0 , \infty ; H^2)} + \sup\limits_{\tau \geqslant 0}	\| \theta \|_{L^2(\tau , \tau + 1 ; H^3)}
+ \| \theta_t \|_{L^{\infty}(0 , \infty ; L^2)} \\
&\quad + \sup\limits_{\tau \geqslant 0} \| \nabla \theta_t \|_{L^2(\tau , \tau + 1 ; \boldsymbol{L}^2)} \leqslant C,
\end{align*}
where $C>0$ is a constant independent of time.
\end{remark}
\medskip 

\begin{remark}
The results in Theorems \ref{2dweak-exist}--\ref{2dstrong} can be extended to the case with a more general singular potential $W$. For instance, the following admissible setting is well-known: 
 $W\in C([-1,1])\cap C^{3}(-1,1)$ such that $W(r)=F(r)-\frac{B}{2}r^2$ with
	\begin{equation}
	\lim_{r\to \pm 1} F'(r)=\pm \infty ,\quad \text{and}\ \  F''(r)\ge A,\quad \forall\, r\in (-1,1). \notag 
	\end{equation}
Here, $A$ is a strictly positive constant, $B\in \mathbb{R}$ and we make the extension $F(r)=+\infty$ for $r\notin[-1,1]$.
Assume in addition, there exists $\varepsilon_0\in(0,1)$ such that $F''$ is nondecreasing in $[1-\varepsilon_0,1)$ and nonincreasing in $(-1,-1+\varepsilon_0]$. The convex function $F$ satisfies the growth assumption:
\begin{align} F^{\prime \prime}(r) \leq C \mathrm{e}^{C\left|F^{\prime}(r)\right|}, \quad \forall\, r \in(-1,1). \notag
\end{align}
for some positive constant $C$. We refer to \cite{Abels2009,Conti2020,GiorginiGalerkin, GiorginiOnoo,HwuHeleshaw,Giorgini2019,HeWeak,HeStrong} for further details.
\end{remark}

\section{Global Weak Solutions} \label{Globalweak}
\subsection{Existence} \label{Globalweak-exist} \begin{proof} (\textbf{Proof of Theorem \ref{2dweak-exist}.})
We first establish the existence of global weak solutions.
The proof relies on the so-called semi-Galerkin scheme, in which we only approximate the velocity field $\u$, but keep the original form of $\theta$ and the distributional form of $\phi$. This makes sense because $\theta$ admits a maximum principle, which is essential for the estimates.

The proof of Theorem \ref{2dweak-exist} consists of several steps.

\subsubsection{Construction of approximate solutions}
Let $\w_{i}(x)$, $i=1,2, \cdots$, be the eigenfunctions of the Stokes operator subject to the homogeneous Dirichlet boundary condition. We can suppose without loss of generality they form an orthonormal basis of $\L$ and an orthogonal basis of $\V$.
Let $\H_m := \operatorname{span}\left\{\w_1(x), \cdots, \w_m(x)\right\}$. Moreover, define $\Pi_m\L = \H_m$ be the orthonormal projection from $\L$ onto $\H_m$.

Let $T>0$, $\u^m(x, t)=\sum\limits_{i=1}^m g_i^m(t) \w_i(x)$. We consider the approximate system holding for arbitrary $\w^m \in \H_m$, $w \in H^1$:
\begin{align}
(\u_t^m, \w^m) + (\u^m \cdot \nabla \u^m, \w^m)+( 2 \nu(\theta^m) \D \u^m, \nabla \w^m) &= \nonumber\\
\int_{\Omega}[\lambda(\theta^m) \nabla \phi^m \otimes \nabla \phi^m]: \nabla \w^m \,\mathrm{d} x  &+\int_{\Omega} \theta^m \g \cdot \w^m \,\mathrm{d} x,
\label{semiu}\\
(\phi_t^m , w)+(\u^m \cdot \nabla \phi^m, w) + (\nabla \mu^m, \nabla w)&=0,
\label{semiphi}\\
\mu^m = - \Delta \phi^m &+ W^{\prime}(\phi^m),
\label{semimu} \\
\theta_t^m+\u^m \cdot \nabla \theta^m-\nabla \cdot(\kappa(\theta^m) \nabla \theta^m)=0, ~~  &\text{ a.e. in }(0, T) \times \Omega,
\label{semitheta}\\
\left.\u^m\right|_{t=0}=\Pi_{m}\u_0,\ \ \left.\phi^m\right|_{t=0}=\phi_0 ,\ \  \left.\theta^m\right|_{t=0}&=\theta_0,
\label{semiinitial}\\
\left. \partial_\mathbf{n} \phi^m\right|_{\partial \Omega}
= 0,\ \ \left.\theta^m\right|_{\partial \Omega} &= 0.
\label{semiboundary}
\end{align}
\begin{thm}\label{exe-approx}
Let $\u_0 \in \L(\Omega)$, $\phi_0 \in H^{1}(\Omega)$ with $\left\|\phi_0\right\|_{L^{\infty}} \leqslant 1$,
$\left|\overline{\phi_0}\right|<1$, and $\theta_0 \in H_0^{1}(\Omega) \cap L^{\infty}( \Omega)$. For any positive integer $m$, problem (\ref{semiu})-(\ref{semiboundary}) admits a local solution $(\u^m, \phi^m, \mu^m, \theta^m)$ on some interval $[0,T_m]$ with the following regularity properties:
\begin{align*}
\u^m &\in C\left([0, T_m] ; \H_m\right), \\
\phi^m &\in L^{\infty}\left(0, T_m ; H^1\right) \cap L^4\left(0, T_m ; H^2\right) \cap H^1\left(0, T_m ;\left(H^1\right)^{\prime}\right), \\
\mu^m &\in L^2\left(0, T_m ; H^1\right),\\
\theta^m &\in L^{\infty}\left(0, T_m ; H_0^1 \cap L^{\infty}\right) \cap L^2\left(0, T_m ; H^2\right).	
\end{align*}
Moreover, we have $\phi^m \in L^{\infty}(\Omega \times (0,T_m))$, $|\phi^m|<1$ almost everywhere in $\Omega \times (0,T_m),$ and $\sup \limits_{0\leqslant t\leqslant T_m} \| \phi^m(t) \|_{L^{\infty}} \leqslant 1. $
\end{thm}

\medskip
\begin{remark} The existence of a local solution on some interval $[0,T_m]$ to problem (\ref{semiu})-(\ref{semiboundary}) can be guaranteed by a fixed point argument, see Appendix \ref{semigalerkin_weaksolution} for details. Here, we emphasize the key consideration that  $\theta^m$ satisfies the maximum principle  (see \cite{Lorca}):
\begin{align} \label{semimaximum}
\|\theta^m(t)\|_{L^{\infty}} \leqslant \|\theta_0\|_{L^{\infty}},\quad \text{for a.e.} ~ t \in [0,T_m].
\end{align}
\end{remark}

\subsubsection{\emph{a priori} estimates}
We proceed to derive estimates for approximate solutions that are uniform with respect to the approximating parameter $m$. The superscript $m$ is dropped for simplicity. In the subsequent proofs,
$C$ represents a generic constant depending only on $\Omega$ and parameters of the system unless additionally pointed out, and may be different from line to line. \smallskip

\textbf{a). Estimates for $\|\nabla \phi\|$ and $\|\u\|$.} Testing (\ref{semiphi}) by $\mu$, we obtain
\begin{equation} \label{nabla phi m}
	\frac{\mathrm{d}}{\mathrm{d} t}\left(\frac{1}{2} \| \nabla \phi \|^2 +
	\int_{\Omega} W\left(\phi\right)\mathrm{d} x \right)  +
	\left\|\nabla \mu\right\|^2 =
	-\left(\u \cdot \nabla \phi ,\mu \right) =
	(\phi \u , \nabla \mu ).
\end{equation}
Testing (\ref{semiu}) by $\u$, we get
\begin{align}
& \frac{1}{2} \frac{\mathrm{d}}{\mathrm{d} t}  \| \u \|^2  +
\int_{\Omega} 2 \nu(\theta) |\D \u|^2 \mathrm{d}x \notag \\
&\quad = a \lambda_0 \left( \nabla \phi \otimes \nabla \phi , \nabla \u \right) -
b \lambda_0 \left(\theta \nabla \phi \otimes \nabla \phi , \nabla \u \right) + ( \theta \g , \u).
\label{u m}
\end{align}
Besides, we observe that
\begin{align} \label{Kroneckeridentity}
(\nabla \phi \otimes \nabla \phi, \nabla \u) &= - (\nabla \cdot (\nabla \phi \otimes \nabla \phi) , \u) \nonumber \\
&= -\left(\Delta \phi \nabla \phi + \nabla \frac{|\nabla \phi|^2}{2}, \u \right) \nonumber \\
&= \left((\mu - W^{\prime}(\phi)) \nabla \phi , \u \right) \nonumber \\
&= (\mu \nabla \phi, \u) \nonumber \\
&= -(\phi \u, \nabla \mu). 	
\end{align}
Let $\ell>0$ be a small number to be determined later. We multiply (\ref{u m}) by $\ell$, and add it to (\ref{nabla phi m}). Keeping in mind that $\|\phi\|_{L^{\infty}} \leqslant 1$, this yields
\begin{align}
&\frac{\mathrm{d}}{\mathrm{d} t}\left(\frac{1}{2} \| \nabla \phi \|^2 +
\int_{\Omega} W\left(\phi \right)\, \mathrm{d} x + \frac{\ell}{2} \| \u \|^2   \right) +
\left\|\nabla \mu\right\|^2 +\ell
\int_{\Omega} 2 \nu(\theta) |\D \u|^2\, \mathrm{d} x
\notag \\
&\quad = (- a \lambda_0 \ell + 1) (\phi \u, \nabla \mu) -
b \lambda_0 \ell \left(\theta \nabla \phi \otimes \nabla \phi , \nabla \u\right) +
\ell ( \theta \g , \u)
\notag \\
&\quad \leqslant C \| \phi\|_{L^{\infty}} \|\u\| \|\nabla \mu\| +
|b| \lambda_0 \ell \|\theta_0\|_{L^{\infty}} \|\nabla \phi\|_{\boldsymbol{L}^4}^2 \|\nabla \u\| +
C \|\theta_0 \|_{L^{\infty}} \|\u\|
\notag \\
&\quad \leqslant C\|\u\| \|\nabla \mu\| + C\|\u\| +
C |b| \lambda_0 \ell \|\theta_0\|_{L^{\infty}} \| \phi\|_{L^{\infty}} \|\phi\|_{H^2} \| \nabla \u\|
\notag \\
&\quad \leqslant \frac{\ell \underline{\nu}}{2}\|\nabla \u\|^2 +
\frac{1}{2} \|\nabla \mu\|^2 + C\|\u\|^2 + C + \frac{C |b|^2 \lambda_0^2 \|\theta_0\|_{L^{\infty}}^2}{2\underline{\nu} } \ell \| \phi\|_{H^2}^2.
\label{phi m and nabla u m}
\end{align}
Here, we have used the Gagliardo-Nirenberg inequality. Applying Lemma \ref{phipriori} to  (\ref{phi m and nabla u m}), and, by Korn's inequality \eqref{Korn}, we deduce that
\begin{align}
&\frac{\mathrm{d}}{\mathrm{d} t}\left(\frac{1}{2} \| \nabla \phi \|^2 +
\int_{\Omega} W\left(\phi \right)\, \mathrm{d} x + \frac{\ell}{2} \| \u \|^2   \right) +
\frac{1}{2} \left\|\nabla \mu\right\|^2 +
\frac{\ell \underline{\nu}}{2} \|\nabla \u\|^2
\notag \\
&\quad \leqslant C(1 + \|\u\|^2) + \frac{C |b|^2 \lambda_0^2 \|\theta_0\|_{L^{\infty}}^2 }{2\underline{\nu} } \ell \|\nabla \mu\|^2.
\label{phi m and nabla u m 2}
\end{align}
Taking $\ell>0$ sufficiently small, we find
\begin{align} \label{weak estimate 1}
\frac{\mathrm{d}}{\mathrm{d} t}\left(\frac{1}{2} \| \nabla \phi \|^2 +
\int_{\Omega} W\left(\phi \right)\, \mathrm{d} x + \frac{\ell}{2} \| \u \|^2   \right) +
\frac{1}{4} \left\|\nabla \mu\right\|^2 +
\frac{\ell \underline{\nu}}{2} \|\nabla \u\|^2
\leqslant C(1 + \|\u\|^2 ).
\end{align}
Thanks to the fact that $W(\phi)$ is uniformly bounded from below, we can obtain estimates for $\u$, $\phi$, $\mu$ on $[0,T]$ (see \eqref{es1} below) from \eqref{weak estimate 1} and the classical Gronwall's lemma.
\medskip

\textbf{b). Estimates for $\|\theta\|_{H^1}$.}
The equation for the temperature is somewhat independent, and the estimates for $\theta$ has been well-investigated in \cite{HW2017,ZZF2013}. Briefly speaking, we take a transform $\Theta(x,t) = \int_0^{\theta(x,t)} \kappa(s) \,\mathrm{d}s$ to eliminate the temperature-dependent thermal coefficient, and derive estimates for the new variable $\Theta$. Then we recover estimates for $\theta$ from $\Theta$. Eventually, combining these estimates with (\ref{weak estimate 1}), we have
\begin{align} \label{weak estimate 2}
\|\nabla \theta(t)\|^2+\int_0^t\left(\left\|\theta_t(\tau)\right\|^2 +\|\theta(\tau)\|_{H^2}^2\right) \mathrm{d}\tau \leqslant C, \quad \forall\, t \in[0, T].
\end{align}

\subsubsection{Passage to the limit}
From (\ref{weak estimate 1}), (\ref{weak estimate 2}), we find  that the local approximate solution $(\u^m,\phi^m,\mu^m,\theta^m)$ can be extended to the interval $[0,T]$ for any given $T>0$. Moreover, (\ref{weak estimate 1}) implies
\begin{align}
\left\{\begin{array}{c}
	\nabla \phi^m \in L^{\infty}(0,T;\boldsymbol{L}^2), \\
	\nabla \u^m \in L^{2}(0,T;\boldsymbol{L}^2), \\
	\u^m \in L^{\infty}(0,T;\boldsymbol{L}^2),  \\
	\nabla \mu^m \in L^{2}(0,T;\boldsymbol{L}^2), \\
\end{array} \right. \text{are uniformly bounded with respect to $m$,}
\label{es1}
\end{align}
while the estimate (\ref{weak estimate 2}) yields
$$
\left\{\begin{array}{c}
	\nabla \theta^m \in L^{\infty}(0,T;\boldsymbol{L}^2), \\
	\theta^m_t \in L^{2}(0,T;L^2), \\
	\theta^m \in L^{2}(0,T;H^2),
\end{array} \right.  \text{are uniformly bounded with respect to $m$.}
$$
It follows from (\ref{Proof-phi1}) that $\| \Delta \phi^m\|^2 \leqslant -W^{\prime \prime} (\phi^m) \| \nabla \phi^m\|^2 + (\nabla \mu^m, \nabla \phi^m)$, as a consequence, $\|\phi^m\|_{H^2}^2 \leqslant C(1 + \| \nabla \mu^m\|)$. Therefore,  $\phi^m \in L^4(0,T;H^2)$ are uniformly bounded with respect to $m$. Besides, by a similar argument like in \cite{inequality} (see also \cite[Equation (3.13)]{Conti2020}), we find
\begin{align}
| \overline{\mu^m} | + \|W'(\phi^m)\|_{L^1}\leq C(1+\|\nabla \mu^m\|), \label{meanmum}
\end{align}
which together with the Poincar\'e-Wirtinger inequality yields that $\mu^m\in L^2(0,T;H^1(\Omega))$ is uniformly bounded with respect to $m$. Thus, we can apply Lemma \ref{mu and phi} and conclude that $\phi^m \in L^2(0,T;W^{2,p})$ is uniformly bounded with respect to $m$ for any $p\geqslant 2$. Moreover, by comparison, it is easy to check that $\u^m, \phi^m$ are uniformly bounded with respect to $m$ in $H^1\left(0, T ; \boldsymbol{V}_\sigma^{\prime}\right)$, $H^1\left(0, T ; (H^1)^{\prime}\right)$, respectively.

Hence, by a standard compactness argument (see e.g., \cite{compact}), we can pass to the limit as $m$ towards $ +\infty$ (up to a subsequence) in the semi-Galerkin scheme (\ref{semiu})-(\ref{semiboundary}) to obtain a global weak solution of problem (\ref{NSCHM})-(\ref{boundary}) on $[0,T]$. The details are omitted here. Uniqueness of the weak solution is a direct consequence of the continuous dependence estimate \eqref{continuousdependence}.

\subsubsection{Uniform-in-time estimates}
\noindent Assume in addition, $\theta_0 \in C^{\gamma}(\overline{\Omega})$ for some $\gamma \in (0,1)$. We proceed to derive uniform-in-time estimates of the global weak solution to problem (\ref{NSCHM})--(\ref{boundary}).
Similarly, we can check that the estimates \eqref{nabla phi m} and \eqref{u m} are also valid for the weak solution. Then, multiplying (\ref{nabla phi m}) by $a \lambda_0$ and adding it together with (\ref{u m}), we infer from (\ref{Kroneckeridentity}) that
\begin{align}
&\frac{\d}{\dt} \left(
\frac{1}{2} a \lambda_0 \| \nabla \phi\|^2 + a \lambda_0  \int_{\Omega} W(\phi)\, \dx + \frac{1}{2} \| \u\|^2 \right) + \left(  a \lambda_0 \| \nabla \mu\|^2 + \int_{\Omega} 2 \nu(\theta) |\D \u|^2 \, \dx \right)
\notag \\
&\quad= - b \lambda_0 (\theta \nabla \phi \otimes \nabla \phi, \nabla \u)
+  (\theta \g ,\u).
\label{new-es}
\end{align}
The terms on the right-hand side can be estimated by the  Gagliardo-Nirenberg inequality, the Poincar\'e inequality and the $H^2$-estimates for $\phi$. Recalling that $\|\phi\|_{L^\infty}\leq 1$ and $\|\theta\|_{L^{\infty}} \leqslant \|\theta_0\|_{L^{\infty}}$, we have 	
\begin{align}
	- b \lambda_0 (\theta \nabla \phi \otimes \nabla \phi, \nabla \u) &\leqslant
	|b| \lambda_0 \|\theta\|_{L^{\infty}} \| \nabla \phi\|_{\boldsymbol{L}^4}^2 \| \nabla \u\| \notag \\
	&\leqslant \varepsilon \| \nabla \u\|^2 + C\| \phi\|_{L^{\infty}}^2 \| \phi\|_{H^2}^2 \notag \\
	&\leqslant \varepsilon \| \nabla \u\|^2 + C (\| \phi \|^2 + \| \nabla \mu\|\| \nabla \phi\|) \notag \\
	&\leqslant \varepsilon \| \nabla \u\|^2 + \varepsilon \| \nabla \mu\|^2 + \widetilde{C} \| \nabla \phi\|^2 + C.
	\label{another}
\end{align}	
Here, $\widetilde{C}$ is a positive constant depending on $\varepsilon$. Furthermore, the H\"older's and Young's inequalities imply
\begin{align*}
	(\theta \g ,\u) \leqslant C\|\theta\|\| \u \| \leqslant \varepsilon \| \nabla \u\|^2 + C.
\end{align*}	
Hence, we infer from \eqref{new-es} and Korn's inequality \eqref{Korn} that
\begin{align}
&\frac{\d}{\dt} \left(
\frac{1}{2} a \lambda_0 \| \nabla \phi\|^2 + a \lambda_0  \int_{\Omega} W(\phi) \dx + \frac{1}{2} \| \u\|^2  \right)
+    (a \lambda_0 - \varepsilon) \| \nabla \mu\|^2 + (\underline{\nu} - 2\varepsilon) \|\nabla \u\|^2
 \notag \\
&\quad \leqslant C + \widetilde{C}\| \nabla \phi\|^2.
\label{mainenergy1}
\end{align}	
To control the last term on the right hand side, we test (\ref{CH2}) by $\phi - \overline{\phi}$, and find
\begin{align}
	\| \nabla \phi\|^2 + \int_{\Omega} W^{\prime}(\phi) (\phi - \overline{\phi})\, \dx &= \int_{\Omega} \mu (\phi - \overline{\phi})\, \dx \notag \\
	&= \int_{\Omega} (\mu - \overline{\mu})\phi\, \dx \notag \\
	&\leqslant C\| \nabla \mu\| \|\phi\| \notag \\
	&\leqslant \varepsilon \| \nabla \mu\|^2 + C.
\label{subordinate}
\end{align}
Multiplying (\ref{subordinate}) by $2\widetilde{C}$, adding together with (\ref{mainenergy1}), and taking $\varepsilon$ sufficiently small, we have
\begin{align}
	&\frac{\d}{\dt} \left(
	\frac{1}{2} a \lambda_0 \| \nabla \phi\|^2 + a \lambda_0  \int_{\Omega} W(\phi) \dx + \frac{1}{2} \| \u\|^2  \right) + \widetilde{C} \| \nabla \phi\|^2 +
	2\widetilde{C} \int_{\Omega} W^{\prime}(\phi) (\phi - \overline{\phi}) \,\dx \notag \\
	&\quad +   \frac{a \lambda_0}{2} \| \nabla \mu\|^2 + \dfrac{\underline{\nu}}{2} \|\nabla \u\|^2   \leqslant C. \label{mainenergy2}
\end{align}	
Set
$$
\mathcal{E}_1(t) = \dfrac{1}{2} a \lambda_0 \| \nabla \phi(t)\|^2 + a \lambda_0  \int_{\Omega} W(\phi(t)) \dx + \dfrac{1}{2} \| \u(t)\|^2.
$$
For the logarithmic potential $W$ (recalling \eqref{Wphi}), applying Taylor's expansion, we can check that (see e.g., \cite[Inequality (3.2)]{GiorginiOnoo})
$$
\int_{\Omega} W^{\prime}(\phi) (\phi - \overline{\phi}) \,\dx
\geqslant   \int_{\Omega} W(\phi) \, \dx -\int_{\Omega} W(\overline{\phi}) \, \dx - \frac{\alpha}{2} \int_\Omega (\phi - \overline{\phi})^2 \,\dx.
$$
Then by the Sobolev embedding theorem and \eqref{mainenergy2}, we can find some constants $c_1, c_1' > 0$ such that
\begin{align}
\frac{\mathrm{d}}{\mathrm{d}t} \mathcal{E}_1 + c_1 \mathcal{E}_1 + c_1'(\| \nabla \mu\|^2 + \|\nabla \u\|^2) \leqslant C, \notag
\end{align}
which together with Gronwall's lemma entails
\begin{align}
\| \u\|_{L^{\infty}(0, \infty; \L)} +
\sup\limits_{\tau \geqslant 0} \| \u \|_{L^2(\tau, \tau + 1; \V)} +
\| \phi \|_{L^{\infty}(0, \infty; H^1)} +
\sup\limits_{\tau \geqslant 0} \| \nabla \mu\|_{L^2(\tau, \tau + 1; L^2)}
\leqslant C.	\notag
\end{align}
Here, the constant $C$ not only depends on $\Omega$ and parameters of the system, but also depends on norms of the initial data.
By interpolation, we obtain
$$
\sup\limits_{\tau \geqslant 0} \| \u \|_{L^4( \tau, \tau + 1; \boldsymbol{L}^4)} \leqslant C.
$$
Moreover, using \eqref{meanmum}, we can deduce that
\begin{align}
\sup\limits_{\tau \geqslant 0} \| \mu\|_{L^2(\tau, \tau + 1; H^1)}
\leqslant C. \notag
\end{align}

Let us now consider the estimate of $\theta$. First, testing the equation (\ref{BOUSSINESQ}) by $\theta$, we have
\begin{align}
  \frac{1}{2}\frac{\d}{\dt}  \| \theta \|^2   + \int_{\Omega} \kappa(\theta) | \nabla \theta |^2\, \dx = 0.\notag
\end{align}
Thanks to the Poincar\'e inequality, it follows from the above inequality that
\begin{align} \label{thetalow}
\| \theta \|_{L^{\infty} (0 , \infty; L^2)} + \sup\limits_{\tau \geqslant 0} \| \theta\|_{L^2(\tau,\tau + 1; H^1)}
  \leqslant C.
\end{align}
Testing the equation (\ref{BOUSSINESQ}) by $ - \Delta \theta $, we have
\begin{align}
	\frac{1}{2} \frac{\d}{\dt}  \| \nabla \theta \|^2
	+ \int_{\Omega} \kappa(\theta) |\Delta \theta|^2\, \dx=
	\int_{\Omega} (\u \cdot \nabla \theta) \Delta \theta\, \dx- \int_{\Omega} \kappa^{\prime}(\theta) |\nabla \theta|^2 \Delta \theta \,\dx.
	\label{nablatheta}
\end{align}
Using the Gagliardo-Nirenberg inequality, the first term on the right hand side of \eqref{nablatheta} can be estimated as follows
\begin{align}
\int_{\Omega} (\u \cdot \nabla \theta) \Delta \theta \,\dx &=
- \int_{\Omega} \nabla \u : (\nabla \theta \otimes \nabla \theta)\, \dx \notag \\
&\leqslant \|\nabla \u\| \|\nabla \theta\|^2_{\boldsymbol{L}^4} \notag \\
&\leqslant C \|\nabla \u\| \|\theta\|_{L^{\infty}} \| \Delta \theta\| \notag \\
&\leqslant \varepsilon \|\Delta \theta\|^2 + C\| \nabla \u\|^2.
\end{align}
To control the second term, the following lemma will be useful, which can be proved as in \cite[Lemma 3.2]{ZZF2013} by using the known estimate for $\sup\limits_{\tau \geqslant 0} \| \u \|_{L^4\left(\tau, \tau + 1;\boldsymbol{L}^4\right)}$.
\medskip

\begin{lemma}[H\"{o}lder estimate for $\theta$]
Let $(\u,\phi,\mu,\theta)$ be a weak solution to   problem (\ref{NSCHM})--(\ref{boundary}) satisfying $\u_0 \in \L(\Omega)$, $\phi_0 \in H^1(\Omega)$ with  $\left\|\phi_0\right\|_{L^{\infty}} \leqslant 1$, $\left|\overline{\phi_0}\right|<1$, and $\theta_0 \in C^{\gamma}(\overline{\Omega}) \cap H_0^1(\Omega)$ for some $\gamma \in (0,1)$. Then we have
\begin{align}
\|\theta\|_{L^{\infty} (0,\infty; C^{\beta}(\overline{\Omega}))} \leqslant C, \label{thetaalpha}
\end{align}
for some $\beta \in(0, \gamma]$. The positive constant $C$ depends on $\| \theta_0\|_{C^{\gamma}(\overline{\Omega})}$, $\sup\limits_{\tau \geqslant 0} \| \u \|_{L^4\left(\tau, \tau + 1;\boldsymbol{L}^4\right)}$, $\Omega$ and parameters of the system. 
\end{lemma}
\medskip

Using \eqref{thetaalpha} and the interpolation (\ref{GiorginiHolder}) (with $\gamma$ replacing by $\beta$), we find some $\xi \in (\frac{1}{2} , 1)$ such that the following crucial estimate holds
\begin{align}
- \int_{\Omega} \kappa^{\prime}(\theta) |\nabla \theta|^2 \Delta \theta\, \dx &\leqslant
\|\kappa^{\prime}(\theta)\|_{L^{\infty}} \|\nabla \theta\|^2_{\boldsymbol{L}^4} \| \Delta \theta\| \notag \\
&\leqslant C \|\theta\|_{C^{\beta}}^{2 \xi} \| \theta\|_{H^2}^{2-2\xi}\| \Delta \theta\| \notag \\
&\leqslant \varepsilon \| \Delta \theta\|^2 + C.
\end{align}	
In the above estimate, we also use the Poincar\'e inequality and the elliptic estimate for $\theta$ that satisfies the homogeneous Dirichlet boundary condition. Therefore, letting $\varepsilon$ be sufficiently small, we obtain
\begin{align}
\frac{1}{2} \frac{\d}{\dt} \|  \nabla \theta \|^2
+ \frac{\underline{\kappa}}{2} \int_{\Omega} |\Delta \theta|^2\, \dx
\leqslant C + C \| \nabla \u\|^2.
\end{align}	
In view of the estimates $\sup\limits_{\tau \geqslant 0} \| \u \|_{L^2(\tau, \tau + 1; \V)} \leqslant C$, $\sup\limits_{\tau \geqslant 0} \| \theta\|_{L^2(\tau,\tau + 1; H^1)} \leqslant C$, we can apply the uniform Gronwall's lemma and deduce that
\begin{align}
\| \theta \|_{L^{\infty} (0 , \infty; H^1)} + \sup\limits_{\tau \geqslant 0} \| \theta\|_{L^2(\tau,\tau + 1; H^2)}    \leqslant C.
\end{align}
The proof of Theorem \ref{2dweak-exist} is complete.
\end{proof}

\subsection{Continuous dependence and uniqueness} \label{Global Weak Solutions-Uniqueness}
To establish the continuous dependence estimate \eqref{continuousdependence}, we mention that the spatial regularity of $\phi$ from the fourth Cahn-Hilliard equation (comparing with the second order Allen-Cahn equation) brings convenience for the treatment of some nonlinear terms. Also, the estimate (\ref{pressure}) for the $L^4$-norm of the pressure in the Stokes problem plays an important role when we handle the term $I_8$ below. We also recall (\ref{Pre-A0}) and (\ref{Pre-StokesS}) for the definition of operators $A_0$ and $S$ that will be used in the subsequent proof.
\medskip
\begin{proof}(\textbf{Proof of Theorem \ref{2dweak-unique}.})
	Let $(\u_1,\phi_1,\mu_1,\theta_1)$, $(\u_2,\phi_2,\mu_2,\theta_2)$ be two weak solutions to problem (\ref{NSCHM})-(\ref{boundary}) with initial conditions $(\u_{01},\phi_{01},\theta_{01})$, $(\u_{02},\phi_{02},\theta_{02})$, respectively. Write
	\begin{align*}
		(\u,\phi,\mu,\theta) &= (\u_{1} - \u_{2},\phi_{1} - \phi_{2},\mu_{1} - \mu_{2},\theta_{1} - \theta_{2}), \\
		(\u_0,\phi_0,\theta_0)    &= (\u_{01} - \u_{02},\phi_{01} - \phi_{02},\theta_{01} - \theta_{02}).
	\end{align*}
 Clearly, we have the following identities for all $\xi \in H^1(\Omega)$, $\v \in \V$:
	\begin{align}
		&\left\langle \u_t, \v\right\rangle_{\V} +
		\int_{\Omega}\left(\u_1 \cdot \nabla \u  + \u \cdot \nabla \u_2\right) \cdot \v \,\mathrm{d} x
		+\int_{\Omega} 2 \nu\left(\theta_1\right) \D \u: \nabla \v \, \mathrm{d} x\nonumber\\
&\qquad 		+\int_{\Omega} 2 \left[\nu\left(\theta_1\right)-\nu\left(\theta_2\right)\right] \D \u_2:\nabla \v\,  \mathrm{d} x
		\nonumber \\
		&\quad =\int_{\Omega} \left( \lambda\left(\theta_1\right) \nabla \phi_1 \otimes \nabla \phi_1
		-\lambda\left(\theta_2\right) \nabla \phi_2 \otimes \nabla \phi_2 \right): \nabla \v \,\mathrm{d} x
		+\int_{\Omega} \theta \g \cdot \v \,\mathrm{d} x,
		\label{u uniqueness} \\
		&\left\langle\phi_t, \xi\right\rangle_{H^1}+\left(\u_1 \cdot \nabla \phi, \xi\right)+\left(\u \cdot \nabla \phi_2, \xi\right) +(\nabla \mu, \nabla \xi) = 0,
		\label{phi uniqueness} \\
		&\theta_t+ \u_1 \cdot \nabla \theta + \u \cdot \nabla \theta_2 =
		\nabla \cdot \left( \kappa\left(\theta_1\right) \nabla \theta\right)
		+\nabla \cdot \left(\left(\kappa\left(\theta_1\right)-\kappa\left(\theta_2\right)\right) \nabla \theta_2\right).
		\label{theta uniqueness}
	\end{align}
In what follows, $\varepsilon$ represents a small positive constant, while $C$ represent constants that may depend on $\varepsilon$, $\Omega$ and parameters of the system. \medskip
	
  \textbf{Estimates for $\|\phi - \overline{\phi}\|_{V_0^{\prime}}$.} Testing (\ref{phi uniqueness}) by $A_0^{-1} (\phi - \overline{\phi})$, we get
	\begin{align} \label{I1toI2}
		\frac{1}{2} \frac{\mathrm{d}}{\mathrm{d} t}\|\phi - \overline{\phi}\|_{V_0^{\prime}}^2+\left(\mu, \phi - \overline{\phi} \right) = I_1 + I_2,
	\end{align}
	where
	\begin{align} \label{I1I2}
		I_1=\left(\phi \u_1, \nabla A_0^{-1} ( \phi - \overline{\phi} )  \right), \quad
		I_2=\left(\phi_2 \u, \nabla A_0^{-1} ( \phi - \overline{\phi} )  \right).
	\end{align}
In light of the conservation of mass $\overline{\phi} = \overline{\phi_0}$, we infer that
	\begin{align}
		(\mu,\phi - \overline{\phi}) &=(-\Delta \phi+W^{\prime}(\phi_1)-W^{\prime}(\phi_2), \phi - \overline{\phi}) \notag \\
		&=\|\nabla \phi\|^2 +
		(W^{\prime}(\phi_1)-W^{\prime}(\phi_2), \phi - \overline{\phi})
		\notag \\
		& \geqslant\|\nabla \phi\|^2 -
		\alpha | (\phi , \phi - \overline{\phi})|
		\notag \\
		&= \|\nabla \phi\|^2 -
		\alpha | (\phi - \overline{\phi} , \phi - \overline{\phi})|
		\notag \\
		&=\| \nabla \phi \|^2 - \alpha |\left(\nabla A_0^{-1} (\phi - \overline{\phi}),
		\nabla (\phi - \overline{\phi})  \right)|
		\notag \\
		& \geqslant \|\nabla \phi\|^2 -
		\left(\frac{1}{2} \|\nabla (\phi - \overline{\phi})\|^2 +
		\frac{\alpha^2}{2} \| \phi - \overline{\phi} \|_{V_0^{\prime}}^2\right)
		\notag \\
		&=\frac{1}{2}\|\nabla \phi\|^2-\frac{\alpha^2}{2}\| \phi - \overline{\phi} \|_{V_0^{\prime}}^2.
		\label{muphi-phi}
	\end{align}
	Combining (\ref{I1toI2}), (\ref{muphi-phi}), we find
	\begin{align}\label{phi-phi}
		\frac{1}{2}\frac{\mathrm{d}}{\mathrm{d} t} \|\phi - \overline{\phi}\|_{V_0^{\prime}}^2 + \frac{1}{2} \| \nabla \phi \|^2  \leqslant
		\frac{\alpha^2}{2} \| \phi - \overline{\phi} \|_{V_0^{\prime}}^2 +I_1+I_2.
	\end{align}
	
 \textbf{Estimates for $\|\theta\|$.}
  Multiplying (\ref{theta uniqueness}) by $\theta$ and integrating over $\Omega$, we get
	\begin{align} \label{theta}
		\frac{1}{2} \frac{\mathrm{d}}{\mathrm{d} t} \| \theta \|^2 +
		\int_{\Omega} \kappa\left(\theta_1\right)|\nabla \theta|^2 \, \mathrm{d} x = -(\u \cdot \nabla \theta_2,\theta) - \int_{\Omega}\left(\kappa\left(\theta_1\right)-\kappa\left(\theta_2\right)\right) \nabla \theta_2 \cdot \nabla \theta \,\mathrm{d} x.
	\end{align}
	Applying the Gagliardo-Nirenberg inequality, we see that
	\begin{align} \label{thetaestimate1}
		-(\u \cdot \nabla \theta_2, \theta) &\leqslant \| \u \| \| \nabla \theta_2 \|_{\boldsymbol{L}^4} \| \theta\|_{L^4} \notag \\
		&\leqslant C \| \u \| \| \theta_2 \|_{L^{\infty}}^{\frac{1}{2}} \| \theta_2 \|_{H^2}^{\frac{1}{2}} \| \theta\|^{\frac{1}{2}} \|\nabla \theta\|^{\frac{1}{2}} \notag \\
		&\leqslant \varepsilon \| \u \|^2 + \varepsilon \| \nabla \theta\|^2 + C \| \theta_2 \|_{H^2}^2 \| \theta \|^2.
	\end{align}
	and
	\begin{align}
		-\int_{\Omega}\left(\kappa\left(\theta_1\right)
		- \kappa\left(\theta_2\right)\right) \nabla \theta_2 \cdot \nabla \theta \,\mathrm{d} x &\leqslant
		\|\kappa^{\prime}\|_{L^{\infty}}  \| \theta\|_{L^4} \|\nabla \theta\| \| \nabla \theta_2 \|_{\boldsymbol{L}^4}\notag \\
		&\leqslant C  \| \theta\|^{\frac{1}{2}}\|\nabla \theta\|^{\frac{3}{2}} \| \theta_2 \|_{L^{\infty}}^{\frac{1}{2}} \| \theta_2 \|_{H^2}^{\frac{1}{2}}
		 \notag \\
		&\leqslant \varepsilon \| \nabla \theta \|^2 + C \| \theta_2 \|_{H^2}^2 \| \theta \|^2.
		\label{thetaestimate2}
	\end{align}
Hence, from  (\ref{theta}), (\ref{thetaestimate1}) and (\ref{thetaestimate2}), we find
	\begin{align} \label{thetaestimate}
		\frac{1}{2} \frac{\mathrm{d}}{\mathrm{d} t}\|\theta\|^2 + \int_{\Omega} \kappa\left(\theta_1\right)|\nabla \theta|^2 \mathrm{d} x \leqslant \varepsilon \|\u\|^2 + 2\varepsilon \| \nabla \theta \|^2 + C \| \theta_2 \|_{H^2}^2 \| \theta \|^2.
	\end{align}
	
 \textbf{Estimates for $\|\u\|_{\V^{\prime}}$.}
 Testing (\ref{u uniqueness}) by $S^{-1} \u$ yields
	\begin{align}
		\frac{1}{2} \frac{\mathrm{d}}{\mathrm{d} t} &\|\u\|_{\V^{\prime}}^2 +
		\int_{\Omega} 2 \nu\left(\theta_1\right) \D \u: \nabla S^{-1} \u \, \mathrm{d} x
		\notag \\
		=&-\int_{\Omega}   2 \left(\nu(\theta_1)  -
		\nu(\theta_2)\right) \D \u_2: \nabla S^{-1} \u \,\mathrm{d} x
		-\left[\left(\u_1 \otimes \u, \nabla S^{-1} \u\right) +
		\left(\u \otimes \u_2, \nabla S^{-1} \u\right)\right]
		\notag \\
		&+ \int_{\Omega}\left[\lambda\left(\theta_1\right) (\nabla \phi_1 \otimes \nabla \phi_1)-\lambda\left(\theta_2\right) (\nabla \phi_2 \otimes \nabla \phi_2) \right] :
		\nabla S^{-1} \u \,\mathrm{d} x
		+ \int_{\Omega} \theta \g \cdot \nabla S^{-1} \u  \,\mathrm{d} x
		\notag \\
		:=& I_3+I_4+I_5+I_6.
		\label{I3toI6}
	\end{align}
Besides, recalling that (see e.g., \cite{Giorgini2019})
$$\nabla \cdot (\nabla^t (S^{-1} \u)) = \nabla(\nabla \cdot (S^{-1} \u)) = 0. $$ 
Integrating by parts, and using the relation $-\Delta S^{-1} \u + \nabla P = \u$, we have
	\begin{align}
		&\int_{\Omega} 2 \nu\left(\theta_1\right) \D \u : \nabla S^{-1} \u \,\mathrm{d} x
		\notag \\
		&\quad = \int_{\Omega} 2 \nabla \u : \nu(\theta_1) \D (S^{-1} \u) \dx
		\notag \\
		&\quad = -2 \left(\u , \nabla \cdot(\nu\left(\theta_1\right) \D (S^{-1} \u) )\right)
		\notag \\
		&\quad =-2 \left(\u , \nu^{\prime}\left(\theta_1\right) \nabla \theta_1 \cdot \D (S^{-1} \u) \right) -
		\left(\u ,\nu(\theta_1) \Delta S^{-1} \u\right) 
		\notag \\
		&\quad = -2 \left(\u , \nu^{\prime}\left(\theta_1\right) \nabla \theta_1 \cdot \D (S^{-1} \u) \right) +
		\int_{\Omega} \nu\left(\theta_1\right)|\u|^2 \,\mathrm{d} x
		-\left(\u, \nu\left(\theta_1\right) \nabla P\right)
		\notag \\
		&\quad = -2 \left(\u , \nu^{\prime}\left(\theta_1\right) \nabla \theta_1 \cdot \D (S^{-1} \u) \right) +
		\int_{\Omega} \nu\left(\theta_1\right)|\u|^2 \, \mathrm{d} x
		+\left(\nu^{\prime}\left(\theta_1\right) \nabla \theta_1 \cdot \u, P\right)
		\notag \\
		&\quad :=-I_7 + \int_{\Omega} \nu(\theta_1)|\u|^2 \mathrm{d} x -I_8.
		\label{I7toI8}
	\end{align}
	
 Therefore, combining (\ref{phi-phi}), (\ref{thetaestimate}), (\ref{I3toI6}), (\ref{I7toI8}), and taking $\varepsilon$ sufficiently small, we can deduce that
	\begin{align}
		\frac{\mathrm{d}}{\mathrm{d} t}& \left( \frac{1}{2} \| \phi(t) - \overline{\phi}\|^2_{V_0^{\prime}}  +  \frac{1}{2} \| \u(t) \|^2_{\V^{\prime}} + \frac{1}{2}\|\theta(t)\|^2         \right) +  \dfrac{\underline{\nu}}{2} \| \u \|^2 + \frac{1}{2} \| \nabla \phi\|^2 + \dfrac{\underline{\kappa}}{2} \|\nabla \theta\|^2
		\nonumber \\
		&\quad \leqslant \frac{\alpha^2}{2}\| \phi - \overline{\phi} \|_{V_0^{\prime}}^2 + C \| \theta_2 \|_{H^2}^2 \| \theta \|^2
		+ \sum_{j=1}^8 I_j.
		\label{All Uniqueness}
	\end{align}
Recalling the conservation of mass, i.e., $\overline{\phi} \equiv \overline{\phi_{0}}$, we infer from the Poincar\'{e}-Wirtinger inequality that
	\begin{align}
		I_1 &= (\phi \u_1 , \nabla A_0^{-1} ( \phi - \overline{\phi} ) )
		\notag \\
		&= ((\phi - \overline{\phi}) \u_1 , \nabla A_0^{-1} ( \phi - \overline{\phi} )) \notag \\
		& \leqslant \|\phi - \overline{\phi} \|_{L^6}\left\|\u_1\right\|_{\boldsymbol{L}^3}\|\phi - \overline{\phi}\|_{V_0^{\prime}}
		\notag \\
		& \leqslant C \|\nabla \phi \|
		\left\|\u_1\right\|_{\boldsymbol{L}^3}\|\phi - \overline{\phi} \|_{V_0^{\prime}}
		\notag \\
		& \leqslant \varepsilon
		\|\nabla \phi \|^2
		+C\left\|\u_{1}\right\|_{\boldsymbol{L}^3}^2\|\phi - \overline{\phi} \|_{V_0^{\prime}}^2 ,
		\label{I1}
	\end{align} 
	and
	\begin{align}
		I_2 &=\left( \phi_2 \u, \nabla A_0^{-1} (\phi-\overline{\phi})  \right)
		\notag \\
		& \leqslant\left\|\phi_2\right\|_{L^{\infty}} \|\u\| \|\phi-\overline{\phi} \|_{V_0^{\prime}}
		\notag \\
		& \leqslant \varepsilon\|\u\|^2+C\|\phi-\overline{\phi} \|_{V_0^{\prime}}^2.
		\label{I2}
	\end{align}
Applying the Gagliardo-Nirenberg inequality and Young's inequality to $I_3$, we find
	\begin{align}
I_3 &= -\int_{\Omega} 2 \left(\nu( \theta_1)-\nu\left(\theta_2\right)\right) \D \u_2: \nabla S^{-1} \u \,\mathrm{d} x
\notag \\
&\leqslant C \| \theta\|_{L^4} \| \nabla \u_2 \| \| \nabla S^{-1} \u\|_{\boldsymbol{L}^4} \notag \\
&\leqslant C \| \theta\|^{\frac{1}{2}} \| \nabla \theta \|^{\frac{1}{2}} \| \nabla \u_2 \| \| \nabla S^{-1} \u \|^{\frac{1}{2}}
\| \u \|^{\frac{1}{2}} \notag \\
&\leqslant \varepsilon \| \u \|^2 + C \| \theta\|^{\frac{2}{3}} \| \nabla \theta \|^{\frac{2}{3}} \| \nabla \u_2 \|^{\frac{4}{3}}
\| \nabla S^{-1} \u \|^{\frac{2}{3}} \notag \\
&\leqslant \varepsilon \| \u \|^2 + \varepsilon \| \nabla \theta \|^2 + C \| \theta \| \| \nabla \u_2\|^2 \| \nabla S^{-1} \u \| \notag \\
&\leqslant \varepsilon \| \u \|^2 + \varepsilon \| \nabla \theta \|^2 + C \| \nabla \u_2\|^2 ( \| \theta\|^2 + \| \nabla S^{-1} \u \|^2 ).
\label{I3}
	\end{align}
Similarly, we get
	\begin{align}
		I_4 &= -\left[\left(\u_1 \otimes \u, \nabla S^{-1} \u\right) +
		\left(\u \otimes \u_2, \nabla S^{-1} \u\right)\right]
		\notag \\
		&\leqslant \left(\left\|\u_1\right\|_{\boldsymbol{L}^4} +
		\left\|\u_2\right\|_{\boldsymbol{L}^4}\right) \|\u\| \| \nabla S^{-1} \u \|_{\boldsymbol{L}^4}
		\notag \\
		&\leqslant C\left(\|\u_1\|^{\frac{1}{2}}
		\| \nabla \u_1 \|^{\frac{1}{2}} +
		\|\u_2\|^{\frac{1}{2}}
		\| \nabla \u_2 \|^{\frac{1}{2}} \right)
		\|\u\| \| \nabla S^{-1} \u\|^{\frac{1}{2}} \|\u\|^{\frac{1}{2}}
		\notag \\
		&\leqslant \varepsilon\|\u\|^2+C\left(\left\|\u_1\right\|_{\boldsymbol{H}^{1}}^2+\left\|\u_2\right\|_{\boldsymbol{H}^{1}}^2\right) \| \u \|^2_{\V^{\prime}}.
		\label{I4}
	\end{align}
	The term $I_5$ can be estimated in a standard manner such that
	\begin{align}
		I_5 &= \int_{\Omega}\left(\lambda (\theta_1) \nabla \phi_1 \otimes \nabla \phi_1 - \lambda(\theta_2) \nabla \phi_2 \otimes \nabla \phi_2\right) : \nabla S^{-1} \u \,\mathrm{d} x
		\notag \\
		&=\int_{\Omega} \lambda(\theta_2) \left( \nabla \phi_1 \otimes \nabla \phi_1 -
		\nabla \phi_2 \otimes \nabla \phi_2 \right) : \nabla S^{-1} \u \, \mathrm{d} x
		\notag \\
		&\quad  +\int_{\Omega}\left(\lambda\left(\theta_1\right)-\lambda(\theta_2)\right) \nabla \phi_1 \otimes \nabla \phi_1: \nabla S^{-1} \u \,\mathrm{d} x
		\notag \\
		&=\int_{\Omega} \lambda(\theta_2) \left( \nabla \phi_1 \otimes \nabla \phi +
		\nabla \phi \otimes \nabla \phi_2 \right) : \nabla S^{-1} \u \, \mathrm{d} x
		\notag \\
		&\quad  +\int_{\Omega}\left(\lambda\left(\theta_1\right)-\lambda(\theta_2)\right) \nabla \phi_1 \otimes \nabla \phi_1: \nabla S^{-1} \u \,\mathrm{d} x
		\notag \\
		&\leqslant \left\|\lambda\left(\theta_2\right)\right\|_{L^{\infty}}\left(\left\|\nabla \phi_1\right\|_{\boldsymbol{L}^{\infty}} + \left\|\nabla \phi_2\right\|_{\boldsymbol{L}^{\infty}}\right)\|\nabla \phi\| \|\u\|_{\V^{\prime}}
		\notag \\
		&\quad  + \| \lambda^{\prime}\|_{L^{\infty}} \| \theta \|_{L^6} \| \nabla \phi_1\|_{\boldsymbol{L}^6}^2 \|\u\|_{\V^{\prime}}
		\notag \\
		&\leqslant \varepsilon \| \nabla \phi\|^2 +
		\varepsilon \|\nabla \theta\|^2 + C(
		\|\nabla \phi_1\|^2_{\boldsymbol{L}^{\infty}} +  \|\nabla \phi_2\|^2_{\boldsymbol{L}^{\infty}} +
		\|\nabla \phi_1\|^4_{\boldsymbol{L}^6}  ) \|\u\|_{\V^{\prime}}^2
		\notag \\
		&\leqslant \varepsilon \| \nabla \phi\|^2
        +    \varepsilon \|\nabla \theta\|^2 + C(
		\|\nabla \phi_1\|^2_{\boldsymbol{L}^{\infty}} +  \|\nabla \phi_2\|^2_{\boldsymbol{L}^{\infty}} +
		\|\phi_1\|^4_{H^2}  ) \|\u\|_{\V^{\prime}}^2.
		\label{I5}
	\end{align}	
Moreover, we find
	\begin{align}
		I_6 &= \int_{\Omega} \theta \g \cdot \nabla S^{-1} \u\, \mathrm{d} x
		\leqslant C\|\theta\|^2 + \|\u\|_{\V^{\prime}}^2.
		\label{I6}
	\end{align}
Thanks to the Gagliardo-Nirenberg inequality, estimate \eqref{Pre-Stokesestimate} and Korn's inequality \eqref{Korn}, we have
	\begin{align}
		I_7 &= 2 \left( \u , \nu^{\prime}\left(\theta_1\right) \nabla \theta_1 \cdot \D S^{-1} \u\right)
		\notag \\
		&\leqslant C\|\u\|\left\|\D S^{-1} \u \right\|_{\boldsymbol{L}^4}      \left\|\nabla \theta_1\right\|_{\boldsymbol{L}^{4}}
		\notag \\
		&\leqslant C\|\u\|    \left\|\nabla S^{-1} \u \right\|^{\frac{1}{2}}
		\left\| \u \right\|^{\frac{1}{2}}
		\left\|\nabla \theta_1\right\|^{\frac{1}{2}}
		\left\|\theta_1\right\|_{H^2}^{\frac{1}{2}}
		\notag \\
		&\leqslant \varepsilon\|\u\|^2+C\|\theta_1\|_{H^2}^2 \|\u\|_{\V^{\prime}}^2.
		\label{I7}
	\end{align}
Finally, to handle the term $I_8$, we apply (\ref{pressure}) for the pressure and obtain
	\begin{align} \label{I8}
		I_8 &= -(\nu^{\prime} (\theta_1) \nabla \theta_1 \cdot \u, P) \notag \\
		&\leqslant \| \nu^{\prime} (\theta_1)\|_{L^{\infty}} \| \nabla \theta_1\|_{\boldsymbol{L}^4} \| \u \| \| P \|_{L^4} \notag \\
		&\leqslant C \| \theta_1\|_{L^{\infty}}^{\frac{1}{2}} \| \theta \|_{H^2}^{\frac{1}{2}} \| \u \| \| \u \|^{\frac{1}{2}} \| \nabla S^{-1} \u \|^{\frac{1}{2}} \notag \\
		&\leqslant \varepsilon \| \u \|^2 + C\| \theta_1\|_{H^2}^2 \| \nabla S^{-1} \u\|^2.
	\end{align}
Collecting the estimates  (\ref{I1})-(\ref{I8}), we can deduce from (\ref{All Uniqueness}) that
	\begin{align} \label{UniquenessEnergy1}
		\frac{\mathrm{d}}{\mathrm{d} t} \Lambda(t) + \mathcal{A}(t) \leqslant \mathcal{J}(t) \Lambda(t),
	\end{align}
	where
	\begin{align}
		\Lambda(t) &= \|\u\|_{\V^{\prime}}^2 + \|\phi - \overline{\phi}\|_{V_0^{\prime}}^2 + \|\theta\|^2,
		\notag \\
		\mathcal{A}(t) &= C(\|\u\|^2 + \| \nabla \phi\|^2 + \| \nabla \theta\|^2),
		\notag
	\end{align}
and
	\begin{align}
		\mathcal{J}(t) & = C (
		1 + \|\u_1\|_{\V}^2 + \|\u_2\|_{\V}^2 + \|\phi_1\|_{W^{2,3}}^2 +
		\|\phi_2\|_{W^{2,3}}^2 + \|\phi_1\|_{H^2}^4 \nonumber\\
 &\quad + \| \theta_1 \|_{H^2}^2 + \| \theta_2 \|_{H^2}^2). \notag
	\end{align}
	In view of Theorem \ref{2dweak-exist}, it is straightforward to check that $\mathcal{J} \in L^1(0,T)$. Then the desired conclusion \eqref{continuousdependence} follows from \eqref{UniquenessEnergy1} and Gronwall's lemma.
\end{proof}

\section{Global Strong Solutions} \label{Global Strong Solutions}
\begin{proof} (\textbf{Proof of Theorem \ref{2dstrong}.})
We still apply the semi-Galerkin scheme (\ref{semiu})-(\ref{semiboundary}) for the existence of global strong solutions. Different from the usual Faedo-Galerkin method, in which approximate solutions are automatically smooth, regularity of solutions to the semi-Galerkin scheme is limited. In the analysis for weak solutions, $\phi^m$, $\mu^m$, $\theta^m$ are not doomed to have enough smoothness. Nevertheless, thanks to the previous results, for instance, \cite{Giorgini Heleshaw,Giorgini2019,HeStrong} for the analysis of $\phi^m$, $\mu^m$, and \cite{ZZF2013,HW2017} for the treatment of $\theta^m$, we can obtain a local approximate solution $(\u^m,\phi^m,\mu^m,\theta^m)$ defined on some interval $[0,T_m]$ with sufficiently high regularity for $\phi^m$, $\mu^m$, $\theta^m$. Further details can be found in Appendix \ref{semigalerkin_strongsolution}.

The proof of Theorem \ref{2dstrong} consists of several steps.

\subsection{Construction of approximate solutions}
Let $T>0$. Consider the approximate velocity field $\u^m(x, t)=\sum\limits_{i=1}^m g_i^m(t) \w_i(x)$ and the semi-Galerkin scheme (\ref{semiu})-(\ref{semiboundary}). Then we have the following result:
\medskip

\begin{thm} \label{strongsemi}
Let $\u_0 \in \V(\Omega)$, $\phi_0 \in H^{2}(\Omega)$  with $\partial_\mathbf{n}\phi_0=0$ on $\partial\Omega$, $\left\|\phi_0\right\|_{L^{\infty}} \leqslant 1$,
$\left|\overline{\phi_0}\right|<1$, $\mu_0 \in H^1(\Omega)$ and $\theta_0 \in  H^2( \Omega) \cap H_0^{1}(\Omega)$. For any positive integer $m$, problem (\ref{semiu})-(\ref{semiboundary}) admits a unique local solution $(\u^m, \phi^m, \mu^m, \theta^m)$ on some interval $[0,T_m]$ with the following regularity properties:
\begin{align*}
\u^m &\in C\left([0, T_m] ; \H_m\right), \\
\phi^m & \in L^{\infty}(0,T_m;H^3) \cap L^2(0,T_m;H^4) \cap H^1(0,T_m;H^1), \\
\mu^m &\in L^{\infty}(0,T_m;H^1) \cap L^2(0,T_m;H^3)\cap H^1(0,T_m;(H^1)^{\prime}),\\
\theta^m &\in L^{\infty}(0,T_m;H^2) \cap L^2(0,T_m;H^3) \cap H^1(0,T_m;H^1).	
\end{align*}
Moreover, there exists $\delta_m \in (0,1)$, such that $\|\phi^m(t)\|_{L^{\infty}} \leqslant 1 - \delta_m$ for all $t\in [0,T_m]$.
\end{thm}

The proof of Theorem \ref{strongsemi} can be found in Appendix \ref{semigalerkin_strongsolution}.

\subsection{Higher-order estimates}
Again, we drop the superscript $m$ in the equations (\ref{semiu})-(\ref{semiboundary}) for the sake of simplicity.
First of all, since now $\theta_0 \in H^2(\Omega) \hookrightarrow C^{\gamma}(\overline{\Omega})$ for some $\gamma > 0$,
the analysis in the weak solution part yields the following uniform estimates
\begin{align}
	&\| \u \|_{L^{\infty}(0, \infty; \L)} + \| \phi\|_{L^{\infty}(0, \infty; H^1)} +\| \theta \|_{L^{\infty}(0, \infty; C^{\beta}(\overline{\Omega}) \cap H^1)} + \sup\limits_{\tau \geqslant 0} \| \u\|_{L^2 ( \tau, \tau + 1 ; \V)} \notag \\
	&\quad +
	\sup\limits_{\tau \geqslant 0} \| \mu \|_{L^2 ( \tau, \tau + 1 ; H^1)} +
	\sup\limits_{\tau \geqslant 0} \| \theta \|_{L^2( \tau, \tau + 1 ; H^2)}
	\leqslant C_1,
	\label{loweroderestimate}
\end{align}
for some $\beta \in (0 , \gamma]$.
Here, the positive constant $C_1$ depends on norms of the initial data, $\Omega$ and parameters of the system.
In the following proof, $\varepsilon$ stands for a small positive constant,  and $C$ depends on $\varepsilon$, $\Omega$, parameters of the system and $C_1$ unless additionally pointed out. We stress that $C$ does not depend on time, which is crucial. Also $C$ may have different values from line to line.

\medskip
	
\textbf{Estimate for $\|\Delta \theta\|$.} Higher order estimates for $\theta$ can be achieved by the transform $\Theta(x,t) = \int_0^{\theta(x,t)} \kappa(s) \,\mathrm{d}s$ discussed in Appendix \ref{semigalerkin_weaksolution}. As stated in \cite{HW2017,ZZF2013}, we have
\begin{align} \label{Deltatheta}
\| \Delta \theta\| \leqslant C(\| \theta_t\| + \|\nabla \u \|).
\end{align}

  \textbf{Estimate for $\|\nabla \u\|$.} Testing equation (\ref{semiu}) by $\u_t$, we have	
\begin{align}
\frac{\mathrm{d}}{\mathrm{d} t} \int_{\Omega} \nu(\theta)|\D \u|^2  \mathrm{d} x
+\left\|\u_t\right\|^2 =&
-\int_{\Omega}(\u \cdot \nabla \u)\cdot \u_t \,\mathrm{d} x +
\int_{\Omega} \nu^{\prime}(\theta) \theta_t | \D \u|^2 \,\mathrm{d} x
\notag \\
&-\int_{\Omega} \nabla \cdot \left( \lambda(\theta) \nabla \phi \otimes \nabla \phi \right) \cdot \u_t \,\mathrm{d} x +
\int_{\Omega} \theta \g \cdot \u_t \,\mathrm{d} x
\notag \\
:=& K_1+K_2+K_3+K_4. \label{K1toK4}
\end{align}
The terms $K_1$, $K_2$ and $K_4$ can be estimated by using the Gagliardo-Nirenberg inequality, H\"older's inequality, Young's inequality  and \eqref{loweroderestimate}:
\begin{align}
K_1  \leqslant& \left\|\u_t\right\| \| \u\|_{\boldsymbol{L}^4} \|\nabla \u \|_{\boldsymbol{L}^4}
\notag \\
\leqslant& C \left\|\u_t\right\| \| \u\|^{\frac{1}{2}} \|\nabla \u\|^{\frac{1}{2}} \|\nabla \u\|^{\frac{1}{2}}   \|\Delta \u\|^{\frac{1}{2}}
\notag \\
\leqslant& \varepsilon\left\|\u_t\right\|^2 + C\|\nabla \u\|^2  \|\Delta \u\|
\notag \\
\leqslant& \varepsilon\left\|\u_t\right\|^2+\varepsilon\|\Delta \u\|^2+C\|\nabla \u\|^4,
\label{K1}
\end{align}
\begin{align}
K_2  \leqslant& C \|\nu^{\prime}(\theta)\|_{L^{\infty}} \| \theta_t \| \|\nabla \u\|_{\boldsymbol{L}^4}^2
\notag \\
\leqslant& C \|\theta_t\| \|\nabla \u\| \|\Delta \u\|
\notag \\
\leqslant&  \varepsilon\|\Delta \u\|^2 + C\|\theta_t\|^2 \|\nabla \u\|^2,
\end{align}
\begin{align}
K_4 \leqslant& C\|\theta\|\left\|\u_t\right\| \leqslant \varepsilon \left\|\u_t\right\|^2+C\|\theta\|^2.
\label{K4}	
\end{align}
To estimate the term $K_3$, we calculate the divergence, and treat the two terms respectively. Taking advantage of Lemma \ref{mu and phi}, Lemma \ref{phipriori} (here we use both (\ref{Inequality-phi1}), (\ref{Inequality-phi2})) and (\ref{Deltatheta}), we have
\begin{align}
K_3 =& -\int_{\Omega} \lambda^{\prime}(\theta) \nabla \theta \cdot(\nabla \phi \otimes \nabla \phi) \cdot \u_t\,  \mathrm{d} x
-\int_{\Omega} \lambda(\theta)\left(\Delta \phi \nabla \phi+\nabla \phi \cdot \nabla^2 \phi\right) \cdot \u_t \,\mathrm{d} x
\notag \\
\leqslant& \varepsilon\left\|\u_t\right\|^2 + C\left(\left\|\lambda^{\prime}(\theta)\right\|_{L^ {\infty}}\| \nabla \theta\|_{\boldsymbol{L}^4}\|\nabla \phi\|_{\boldsymbol{L}^8}^2\right)^2 +
C \left(\|\lambda(\theta)\|_{L^{\infty}}\|\phi\|_{W^{2,4}} \|\nabla \phi\|_{\boldsymbol{L}^4} \right)^2
\notag \\
\leqslant & \varepsilon\left\|\u_t\right\|^2 + C\|\Delta \theta\|^2 \| \phi \|_{H^2}^4 + C(1 + \| \nabla \mu\|^2) \| \phi \|_{H^2}^2
\notag \\
\leqslant &  \varepsilon\left\|\u_t\right\|^2 + C\| \Delta \theta\|^2 (\|\phi\|^4 + \| \nabla \mu\|^2) +
C(1 + \|\nabla \mu\|^2) (\|\phi\|^2 + \|\nabla \mu\|^2)
\notag \\
\leqslant & \varepsilon\left\|\u_t\right\|^2 + C\| \theta_t\|^4 + C\| \nabla \u\|^4 + C(\|\phi\|^2 +\|\phi\|^4 + \|\phi\|^8 + \|\nabla \mu\|^2 + \| \nabla \mu\|^4)
\notag \\
\leqslant & \varepsilon\left\|\u_t\right\|^2 + C\| \theta_t\|^4 + C\| \nabla \u\|^4 + C\|\phi\|^2 + C \|\nabla \mu\|^2 + C \|\nabla \mu\|^4.
\label{K3}
\end{align}
\medskip

\textbf{Estimate for $\|\Delta \u\|$.} For the Navier-Stokes equations with temperature-dependent (and thus, space-dependent) viscosity, we recall the following estimates (see e.g., \cite{ZZF2013}):

\medskip

\begin{lemma} \label{stokes}
Let $\Omega\subset \mathbb{R}^2$ be a bounded domain with smooth boundary. Consider the boundary value problem
\begin{equation} \label{stokesequation}
\begin{cases}
&-\nabla \cdot (2 \nu(x) \D \u) + \nabla p = \mathbf{f}, \quad \text{in}\  \Omega,  \\
&\nabla \cdot \u = 0, \quad \text{in}\ \Omega, \\
&\u = 0, \quad \text{on}\ \partial \Omega.
\end{cases}
\end{equation}
If $\mathbf{f} \in \boldsymbol{L}^2(\Omega)$, then problem \eqref{stokesequation} admits a unique solution $(\u,p) \in \boldsymbol{H}^2(\Omega) \times H^1(\Omega)$ with $\overline{p}=0$, such that the following inequality holds:
\begin{align}
&\|\u\|_{\V} + \|p\| \leqslant C\|\mathbf{f}\|_{\boldsymbol{H}^{-1}},  \label{Stokesestimate1}
\\
&\|\u\|_{\boldsymbol{H}^2} + \| \nabla p \| \leqslant C\big(\|\mathbf{f}\| + (1 + \|\nu\|_{H^1} \|\nu\|_{H^2}) \| \nabla \u\| + \|p\|\big),
\label{Stokesestimate2}
\end{align}
where the positive constant $C$ only depends on $\Omega$ and $\nu$.
\end{lemma}

~
	
\noindent We write equation (\ref{semiu}) in the sense of distribution as
\begin{align} \label{Stokesform}
-\nabla \cdot( 2 \nu(\theta(x)) \D \u) + \nabla p = \mathbf{f},
\end{align}
where
\begin{align} \label{Stokesf}
\mathbf{f} = -\u_t-\u \cdot \nabla \u-\nabla \cdot\left[\lambda(\theta) \nabla \phi \otimes \nabla \phi+\lambda(\theta)\left(\frac{1}{2}|\nabla \phi|^2+W(\phi)\right) \mathbb{I}_2\right] + \theta \g .
\end{align}
Inspired by \cite{HW2017}, we proceed to estimate $\| \mathbf{f} \|$ term by term. In \cite{HW2017}, the author considered a phase function $\phi$ satisfying the Allen-Cahn equation. Here, the Cahn-Hilliard equation yields better spatial regularity of $\phi$, and thus leads to an improved estimate for $\| \mathbf{f} \|$. As a consequence, it holds
\begin{align}
\|\mathbf{f}\| \leqslant & \left\|\u_{t}\right\| +
\|\u\|_{\boldsymbol{L}^4} \|\nabla \u\|_{\boldsymbol{L}^4} + C\|\theta\|
\notag \\
&+  \| \lambda (\theta)\|_{L^{\infty}} \|\phi\|_{W^{2,4}}
\|\nabla \phi\|_{\boldsymbol{L}^4}
+ \left\|\lambda^{\prime}(\theta)\right\|_{L^{\infty}} \|\nabla \theta\|_{\boldsymbol{L}^4}
\|\nabla \phi \|_{\boldsymbol{L}^8}^2
\notag \\
&+ C \|\lambda^{\prime}(\theta)\|_{L^{\infty}}
\|\nabla \theta\| \| W(\phi)\|_{L^{\infty}}
+ \|\lambda(\theta)\|_{L^{\infty}} \|W^{\prime}(\phi)\|_{L^4}
\| \nabla \phi \|_{\boldsymbol{L}^4}
\notag \\
\leqslant & \left\|\u_t\right\|+C \| \u \| \|\nabla \u\|^{\frac{1}{2}} \|\nabla \u\|^{\frac{1}{2}}
\|\Delta \u\|^{\frac{1}{2}} + C\|\theta\|
\notag \\
&+ C(1 + \| \nabla \mu\|) \| \phi \|_{H^2}
+ C \|\Delta \theta\| \| \phi \|_{H^2}^2 + C \| \nabla \theta \|
\notag \\
\leqslant & \varepsilon \| \Delta \u\| + C\| \nabla \u\|^2 + \|\u_t\| + C\|\nabla \theta\|
\notag \\
&+ C(1 + \| \nabla \mu\|) (\|\phi\| + \|\nabla \mu\|) +
C(\|\theta_t\| + \|\nabla \u\|) (\|\phi\|^2 +\|\nabla \mu\|)
\notag \\
\leqslant & \varepsilon \| \Delta \u\| + C\| \nabla \u\|^2 + \|\u_t\| + C\|\nabla \theta\|
\notag \\
&+ C\left(\| \phi \| + \|\phi\|^2  + \|\nabla \mu\|^2 \right)+
C(\|\theta_t\| + \|\nabla \u\|) (\|\phi\|^2 +\|\nabla \mu\|),
\label{Stokesfestimate}
\end{align}
where we have used Lemma \ref{mu and phi} to handle $\|\phi\|_{W^{2,4}}$ and $\|W'(\phi)\|_{L^4}$.
In order to apply Lemma \ref{stokes}, we also need to estimate $\|\nu(\theta)\|_{H^1}$ and $\|\nu(\theta)\|_{H^2}$:
\begin{align}
\|\nu(\theta)\|_{H^1} &= \| \nu(\theta) \| + \| \nu^{\prime} (\theta) \nabla \theta \|
\leqslant C(1 + \|\nabla \theta \|), \label{nuh1}\\
\|\nu(\theta)\|_{H^2} &= \|\nu(\theta)\|_{H^1} + \| \nu^{\prime \prime}(\theta) |\nabla \theta|^2 \| +
\| \nu^{\prime} (\theta) \nabla^2 \theta \| \notag \\
&\leqslant \|\nu(\theta)\|_{H^1} + C \| \nabla \theta \|_{\boldsymbol{L}^4}^2 + C \|\Delta \theta \| \notag \\
&\leqslant C(1 + \|\nabla \theta \|) + C\| \theta\|_{L^{\infty}} \| \Delta \theta \| + C \|\Delta \theta \| \notag \\
&\leqslant  C(1 + \|\Delta \theta \|), \label{nuh2}
\end{align}
where we have used the Gagliardo-Nirenberg inequality and standard elliptic estimates. Combining estimates (\ref{Deltatheta}), (\ref{nuh1}), (\ref{nuh2}), we have
\begin{align}
\|\nu(\theta)\|_{H^1} \|\nu(\theta)\|_{H^2}	&\leqslant C(1 + \|\nabla \theta\| + \|\Delta \theta\| +
\|\nabla \theta\| \|\Delta \theta\|) \notag \\
& \leqslant C(1 + \|\Delta \theta\|)  \notag \\
& \leqslant C(1 + \|\theta_t\| + \|\nabla \u\|). \label{nuh1nuh2}
\end{align}
Thus, we can deduce from Lemma \ref{stokes} and (\ref{Stokesestimate1}), (\ref{Stokesestimate2}) that
\begin{align}
\| \Delta \u\| &\leqslant C\big( \| \mathbf{f} \| + (1 + \|\nu(\theta)\|_{H^1} \|\nu(\theta)\|_{H^2}) \| \nabla \u \| + \| p \|\big) \nonumber \\
&\leqslant C\left( \varepsilon \|\Delta \u \| + \|\u_t\| + \|\nabla \u\| + \|\nabla \u\|^2 + \| \nabla \theta\| + \|\phi\| + \|\nabla \mu\|^2 + \|\theta_t\|^2 \right).
\end{align}
Taking $\varepsilon$ sufficiently small, we get
\begin{align} \label{Stokesuestimate}
\| \Delta \u\| &\leqslant C\left(\|\u_t\| + \|\nabla \u\| + \|\nabla \u\|^2 + \| \nabla \theta\| + \|\phi\| + \|\nabla \mu\|^2 + \|\theta_t\|^2 \right). 	
\end{align}
The above estimates together with \eqref{K1toK4} entail that
\begin{align}
&  \frac{\mathrm{d}}{\mathrm{d} t}  \int_{\Omega} \nu(\theta) |\D \u|^2 \,\mathrm{d} x   +
\left[1 - (C+3)\varepsilon \right] \|\u_t\|^2 \nonumber \\
&\quad \leqslant C\big(\|\nabla \u\|^4 + \|\theta_t\|^4 + \|\nabla \mu\|^4 \big)
 + C \big( \|\nabla \u\|^2 + \|\nabla \theta\|^2 + \|\phi\|^2 + \| \nabla \mu\|^2\big).
\label{Higherestimate:u}
\end{align}
\smallskip

  \textbf{Estimates for $\|\nabla \mu\|$.}  Testing equation (\ref{semiphi}) by $\mu_t \in L^2(0,T; (H^1)^{\prime})$, we get 	
\begin{align} \label{nablamu}
\frac{1}{2} \frac{\mathrm{d}}{\mathrm{d} t} \| \nabla \mu\|^2 +
\langle \mu_t,\phi_t \rangle_{H^1} +
\langle  \mu_t,\u \cdot \nabla \phi \rangle_{H^1} = 0,
\end{align}
where
\begin{align} \label{nablamu1}
\left\langle\mu_t,\u \cdot \nabla \phi \right\rangle_{H^1} = \frac{\mathrm{d}}{\mathrm{d} t}\left(\u \cdot \nabla \phi, \mu \right) - \left(\u_t \cdot \nabla \phi, \mu \right)
-\left(\u \cdot \nabla \phi_t, \mu\right).
\end{align}		
Using the Sobolev embedding theorem and the inequality (recall \eqref{meanmum})
\begin{align}
\| \mu \|_{H^1} \leqslant C ( 1 + \| \nabla \mu \|),
\label{muH1}
\end{align}
 we have
\begin{align}
\left(\u_t \cdot \nabla \phi, \mu\right) & \leqslant
\|\nabla \phi\|{ }_{\boldsymbol{L}^3}
\left\|\u_t\right\|
\| \mu \|_{L^6}
\notag \\
& \leqslant \varepsilon\left\|\u_t\right\|^2 +
C\|\phi\|_{H^2}^2\left(1+\left\|\nabla \mu\right\|^2\right)
\notag \\
& \leqslant \varepsilon\left\|\u_t\right\|^2 + C (\|\phi\|^2 + \|\nabla \mu\|^2)
( 1 + \| \nabla \mu \|^2  )
\notag \\
& \leqslant \varepsilon\left\|\u_t\right\|^2 + C (\|\phi\|^2 + \|\phi\|^4 + \| \nabla \mu\|^2 + \| \nabla \mu\|^4 ),
\label{nablamu2}
\end{align}	
and
\begin{align} \label{nablamu3}
(\u \cdot \nabla \phi_t, \mu) &\leqslant
\|\u\|_{\boldsymbol{L}^4}
\| \nabla \phi_t\|
\|\mu\|_{L^4} \notag \\
&\leqslant C \| \u \|^{\frac{1}{2}} \| \nabla \u \|^{\frac{1}{2}} \| \nabla \phi_t \| \| \mu \|^{\frac{1}{2}} (1 + \| \nabla \mu\|)^{\frac{1}{2}} \notag \\
&\leqslant \varepsilon \| \nabla \phi_t\|^2 + C \| \nabla \u\| \| \mu \| (1 + \| \nabla \mu \| ) \notag \\
&\leqslant \varepsilon \| \nabla \phi_t\|^2 + C \| \mu \|^2 + \| \nabla \u\|^2 (1 + \| \nabla \mu\|^2) \notag \\
&\leqslant \varepsilon \| \nabla \phi_t\|^2 + C \| \mu \|^2 + \| \nabla \u\|^2 + \| \nabla \u\|^4 + \| \nabla \mu\|^4.
\end{align}	
Applying the interpolation (\ref{Pre-V0interpolation}) to $\|\phi_t\|$ yields
\begin{align}
\left \langle  \mu_t,\phi_t \right \rangle_{H^1} &=\left\|\nabla \phi_t\right\|^2+\left(W^{\prime \prime}(\phi)  \phi_t, \phi_t\right)
\notag \\
& \geqslant \left\|\nabla \phi_t\right\|^2-\alpha\left\|\phi_t\right\|^2
\notag \\
& \geqslant \left\|\nabla \phi_t\right\|^2 - \left(  \frac{1}{2} \| \nabla \phi_t\|^2 +
\frac{\alpha^2}{2} \| \phi_t\|_{(H^1)^{\prime}}^2 \right)
\notag \\
& \geqslant \frac{1}{2} \|\nabla \phi_t\|^2 - \frac{\alpha^2}{2} \| \phi_t\|_{(H^1)^{\prime}}^2.
\label{nablamu4}
\end{align}	
Besides, from equation (\ref{semiphi}), it is not difficult to derive the estimate (see e.g., \cite{longtime2022}):	
\begin{align} \label{phit}
\| \phi_t \|_{(H^1)^{\prime}} \leqslant C ( \|\u\| + \|\nabla \mu\| ).
\end{align}
Combining \eqref{nablamu2}-(\ref{phit}) and taking $\varepsilon$ sufficiently small, we infer from (\ref{nablamu}) that
\begin{align}
&\frac{\mathrm{d}}{\mathrm{d} t} \left( \frac{1}{2} \|\nabla \mu\|^2 + (\u \cdot \nabla \phi, \mu) \right) + \left( \frac{1}{2} - \varepsilon \right) \| \nabla \phi_t\|^2
\notag \\
& \quad \leqslant \varepsilon \|\u_t\|^2 + C\big( \| \nabla \u \|^2 + \| \nabla \u\|^4 + \|\mu\|^2 + \|\phi\|^2 + \|\nabla \mu\|^2 + \|\nabla \mu\|^4\big).	
\label{nablamu and phit}
\end{align}

\textbf{Estimates for $\|\theta_t\|$.} Differentiating equation (\ref{semitheta}) by $t$, testing the resultant by $\theta_t$, and using the Gagliardo-Nirenberg inequality for $\theta$, we have
\begin{align}
& \frac{1}{2} \frac{\mathrm{d}}{\mathrm{d} t}\left\|\theta_t\right\|^2  +
\int_{\Omega} \kappa(\theta)|\nabla \theta_t|^2 \,\mathrm{d} x \notag\\
&\quad =
-\int_{\Omega} \kappa^{\prime}(\theta) \theta_t \nabla \theta \cdot \nabla \theta_t \,\mathrm{d} x -\int_{\Omega}\left(\u_t \cdot \nabla \theta\right) \theta_t \,\mathrm{d} x - \underbrace{ \int_{\Omega} (\u \cdot \nabla \theta_t) \theta_t \dx }_{=0}
\notag \\
&\quad \leqslant \| \kappa^{\prime}(\theta) \|_{L^{\infty}} \| \nabla \theta_t\| \|\theta_t\|_{L^4} \| \nabla \theta \|_{\boldsymbol{L}^4} + \|\u_t\| \|\nabla \theta\|_{\boldsymbol{L}^4} \|\theta_t\|_{L^4}
\notag \\
&\quad \leqslant \varepsilon \|\nabla \theta_t\|^2 + \varepsilon \|\u_t\|^2 +
C\| \nabla \theta_t\| \|\theta_t\| \|\theta\|_{L^{\infty}} \|\Delta \theta\|
\notag \\
&\quad \leqslant 2\varepsilon \|\nabla \theta_t\|^2 + \varepsilon \|\u_t\|^2 + C\|\theta_t\|^2 \|\Delta \theta\|^2 \notag \\
&\quad \leqslant 2\varepsilon \|\nabla \theta_t\|^2 + \varepsilon \|\u_t\|^2 + C\|\theta_t\|^4 + C\| \nabla \u\|^4.
\label{thetat}
\end{align}
Here, we have also used the estimate (\ref{Deltatheta}).

Combining all the estimates above, we can deduce that
\begin{align}
\frac{\mathrm{d}}{\mathrm{d} t} & \left(
\int_{\Omega} \nu(\theta) | \D \u|^2 \mathrm{d} x
+ \frac{1}{2}\|\nabla \mu\|^2
+ \frac{1}{2}\|\theta_t\|^2 + (\u \cdot \nabla \phi, \mu) \right)
\notag \\
&+  (\underline{\kappa} - 2\varepsilon) \|\nabla \theta_t\|^2 +
\left( \frac{1}{2} - \varepsilon \right) \|\nabla \phi_t\|^2 +
(1 - (C+5) \varepsilon) \|\u_t\|^2
\notag \\
\leqslant & C (\|\nabla \u \|^4 +  \|\theta_t\|^4 + \| \nabla \mu \|^4)
+ C(\| \u \|^2 + \| \nabla \u \|^2 + \|\phi\|^2 + \| \mu \|^2 + \|\nabla \mu\|^2 + \|\nabla \theta\|^2)
\notag \\
\leqslant & C (\|\nabla \u \|^4 +  \|\theta_t\|^4 + \| \nabla \mu \|^4)
+ C(\| \u \|^2_{\V} + \|\phi\|^2 + \| \mu \|^2_{H^1} + \|\theta\|_{H^1}^2),
\label{HigherEstimateAll}
\end{align}	
Define
\begin{align}
\beta(t) &= \int_{\Omega} \nu(\theta) |\D \u|^2 \dx + \frac{1}{2} \|\nabla \mu\|^2
+ \frac{1}{2} \|\theta_t\|^2 + (\u \cdot \nabla \phi, \mu),
\label{higherenergy}\\
\Gamma(t) &=  \frac{\underline{\kappa}}{2} \|\nabla \theta_t\|^2 +
\frac{1}{4} \|\nabla \phi_t\|^2 +
\frac{1}{2} \|\u_t\|^2,
\label{Gammat}\\
\mathcal{G}(t) &= \| \u \|^2_{\V} + \|\phi\|^2 + \| \mu \|^2_{H^1} + \|\theta\|_{H^1}^2.
\label{Gt}
\end{align}	
The coupling term $(\u \cdot \nabla \phi, \mu)$ in \eqref{higherenergy}  does not have a definite sign, but it will not cause any problem since
\begin{align*}
| (\u \cdot \nabla \phi, \mu) | &= | (\u \phi, \nabla \mu) | \\
&\leqslant \|\u\|_{\boldsymbol{L}^4} \| \phi \|_{L^4} \| \nabla \mu\| \\
&\leqslant C \| \u \|^{\frac{1}{2}} \|\nabla \u\|^{\frac{1}{2}} \| \phi \|_{L^4} \| \nabla \mu\| \\
&\leqslant \varepsilon \| \nabla \mu \|^2 +  C \| \phi \|_{L^4}^2 \|\nabla \u\| \\
&\leqslant \varepsilon \| \nabla \mu \|^2 + \varepsilon \|\nabla \u\|^2 + C \| \phi \|_{L^4}^4 \\
&\leqslant \varepsilon \| \nabla \mu \|^2 + \varepsilon \|\nabla \u\|^2 + C \| \phi \|^2.
\end{align*}
Taking $\varepsilon$ sufficiently small, we can deduce from \eqref{HigherEstimateAll} and Korn's inequality \eqref{Korn} that
\begin{align} \label{Higherenergyinequality}
\frac{\mathrm{d}}{\mathrm{d} t} \beta(t) + \Gamma(t) \leqslant C\beta(t)^2 + C\mathcal{G}(t).
\end{align}
Since
$$
\int_{t}^{t+1} \beta (\tau) \,\d \tau + \int_{t}^{t+1} \mathcal{G} (\tau) \,\d \tau \leqslant C, \quad \forall\, t \geqslant 0,
$$
the classical Gronwall's lemma together with the uniform Gronwall's lemma implies that $\beta(t) \leqslant C$ for all $t\geqslant 0$. Recalling \eqref{muH1}, we have
\begin{align}\label{beta}
\| \u(t) \|_{\V}^2 + \| \mu(t) \|_{H^1}^2 + \| \theta_t(t) \|^2 \leqslant C, \quad  \forall\, t \geqslant 0.
\end{align}
Here and hereafter, positive constants $C$ not only depend on $\Omega$, parameters of the system and $C_1$ (see \eqref{loweroderestimate}), but also depend on norms of the initial data.

Concerning $\Gamma(t)$, integrating \eqref{Higherenergyinequality} from $t$ to $t+1$, we obtain
\begin{align}\label{gamma}
\sup\limits_{t \geqslant 0} \int_{t}^{t + 1} \| \nabla \theta_t \|^2 + \| \nabla \phi_t \|^2 + \| \u_t \|^2 \,\d \tau \leqslant C.
\end{align}
The estimate $\sup\limits_{t \geqslant 0} \int_{t}^{t + 1} \| \u \|_{\boldsymbol{H}^2}^2\, \d \tau \leqslant C$ then
follows from (\ref{Stokesuestimate}). Further regularity of $\phi$ and $\mu$ can be deduced like in \cite{HeStrong}. According to   Remark \ref{remark mu and phi}, since $\mu \in L^{\infty}(0, \infty; H^1)$, there exists some small $\delta > 0$, such that
$$
\|\phi(t)\|_{C(\overline{\Omega})} \leqslant 1-\delta, \quad  \forall\, t \geqslant 0.
$$
As a consequence, the singular potential $W^{\prime}(\phi)$ becomes a smooth function. Applying standard elliptic estimates to (\ref{semimu}), we see that $\phi \in L^{\infty}(0 , \infty ; H^3)$. Subsequently, taking the gradient to (\ref{semiphi}), we can prove
$\sup\limits_{t \geqslant 0} \int_{t}^{t + 1} \| \mu \|_{H^3}^2 \,\d \tau \leqslant C$. Applying Laplacian to (\ref{semimu}), we further get
$\sup\limits_{t \geqslant 0} \int_{t}^{t + 1} \| \phi \|_{H^4}^2\, \d \tau \leqslant C$. Next, differentiating (\ref{semimu}) by $t$, by comparison, we can deduce that $\sup\limits_{t \geqslant 0} \int_{t}^{t + 1} \| \mu_t \|_{(H^1)^{\prime}} \,\d \tau \leqslant C$.
Concerning the estimate for $\theta$, we follow the strategy in \cite{HW2017}. Taking the transform $\Theta(x,t) = \int_0^{\theta(x,t)} \kappa(s)\, \mathrm{d}s$, and estimating $\| \Delta \Theta \|$ by using $\| \Theta_t \|$, we have $\Theta \in L^{\infty}(0 , \infty ; H^2)$. Therefore, $\theta \in L^{\infty}(0 , \infty ; H^2)$. To estimate $\| \theta \|_{H^3}$, we can apply gradient to the equation (\ref{semitheta}), and estimate term by term as in the \cite{HW2017}. In this manner, we obtain
$\sup\limits_{t \geqslant 0} \int_{t}^{t + 1} \| \theta \|_{H^3}^2\, \d \tau \leqslant C$.

Collecting the above estimates, we can prove the existence of a global strong solution by the standard compactness argument. Uniqueness again follows from the continuous estimate \eqref{continuousdependence}.

The proof of Theorem \ref{2dstrong} is complete.
\end{proof}

\bigskip

\textbf{Acknowledgments.} The author thanks the anonymous referees for valuable comments. In particular, the author is indebted to Prof. Andrea Giorgini for his constructive suggestions that have greatly improved this paper. The research of the author was partially supported by National Natural Science Foundation of China under Grant number 12071084 (PI: Hao Wu). He acknowledges the support of Prof. Hao Wu and helpful discussions.
\medskip

\appendix

\section{Semi-Galerkin Scheme for Weak Solutions} \label{semigalerkin_weaksolution}
The semi-Galerkin scheme (\ref{semiu})-(\ref{semiboundary}) can be solved by a fixed point argument. In this scheme, the unknown variables $\phi$ and $\theta$ are not approximated, and hence we cannot directly use the ODE theory to derive the existence of approximate solutions. To this end, we first fix some function $\v^m \in C([0, T] ; \H_m)$, and consider the subsystem for $(\phi^m,\mu^m,\theta^m)$. When a solution $(\phi^m,\mu^m,\theta^m)$ is obtained, we solve the subsystem for $\u^m$ and thus construct a mapping from $\v^m$ to $\u^m$. Under suitable assumptions, we will show that the fixed point argument can be applied.

\subsection{Solve $(\phi^m,\mu^m,\theta^m)$ with a given velocity $\v^m$.}
Let $T>0$. We
fix the velocity $\v^m=\sum\limits_{i=1}^m g_i^m(t) \w_i(x)$
$\in C([0, T] ; \H_m)$.
Then we consider the subsystem with convection:
\begin{align}
&\phi_t^m +\v^m \cdot \nabla \phi^m =\Delta \mu^m, \quad \text {in }\Omega \times (0, T), \label{semiphi,vm}\\
&\mu^m = - \Delta \phi^m + W^{\prime}(\phi^m), \quad \text {in }\Omega \times (0, T), \label{semimu,vm}\\
&\theta^m_t+\v^m \cdot \nabla \theta^m-\nabla \cdot(\kappa(\theta^m) \nabla \theta^m)=0, \quad \text {in }\Omega \times (0, T),
\label{semitheta,vm}\\
&\left.\phi^m\right|_{t=0}=\phi_0, \quad \left.\theta^m\right|_{t=0}=\theta_0, \quad \text{in}\ \Omega,
\label{semiinitial,vm}\\
&\left.\theta^m\right|_{\partial \Omega}=0, \left.\quad \partial_\mathbf{n} \phi^m\right|_{\partial \Omega}
= \left.  \partial_\mathbf{n} \mu^m \right|_{\partial \Omega}=0,\quad \text{on}\ \partial\Omega\times (0,T).
\label{semiboundary,vm}
\end{align}
Since the equations for $(\phi^m,\mu^m)$ and $\theta^m$ are decoupled in the above system, we can handle them separately.
\medskip

\textbf{Solvability of $\theta^m$.} Consider the system for $\theta^m$:
\begin{equation} \label{semi:thetamequation}
\begin{cases}
\theta_t^m+\v^m \cdot \nabla \theta^m=\nabla \cdot\left(\kappa\left(\theta^m\right) \nabla \theta^m\right), \quad \text { in } \Omega \times (0, T), \\
\theta^m=0, \quad \text { on } \partial \Omega \times(0, T), \\
\left.\theta^m\right|_{t=0}=\theta_0(x), \quad \text { in } \Omega.
\end{cases}	
\end{equation}
We have  
\medskip

\begin{lemma} \label{semi:thetamlemma}
Assume that $\v^m \in C\left([0, T] ; \H_m\right)$, $\theta_0 \in H_0^1(\Omega) \cap L^{\infty}(\Omega)$. Problem (\ref{semi:thetamequation}) admits a unique solution $\theta^m \in L^{\infty}\left(0, T ; H_0^1(\Omega) \cap L^{\infty}(\Omega)\right) \cap L^2\left(0, T ; H^2(\Omega)\right)$, which satisfies
\begin{align}
& \left\|\theta^m(t)\right\|_{L^{\infty}} \leqslant\left\|\theta_0\right\|_{L^{\infty}}, \quad \forall\, t \in[0, T],
\label{semi:thetammaximum}\\
& \sup _{t \in[0, T]}\left\|\theta^m(t)\right\|^2+2 \underline{\kappa} \int_0^T\left\|\nabla \theta^m(t)\right\|^2 \,\mathrm{d}t \leqslant\left\|\theta_0\right\|^2, \label{semi:thetamenergy1}
\\
& \sup _{t \in[0, T]}\left\|\theta^m(t)\right\|_{H^1}^2 +\int_0^T\left\|\theta^m(t)\right\|_{H^2}^2 \, \mathrm{d}t \leqslant C.
\label{semi:thetamenergy2}
\end{align}
Furthermore, we have the continuous dependence estimate:
\begin{equation} \label{continuous1}
\left\|\theta_1^m(t) - \theta_2^m(t)\right\|^2 +
\int_0^t \| \nabla (\theta_1^m-\theta_2^m)(\tau) \|^2\, \mathrm{d}\tau  \leqslant
C \int_0^t\left\|\v^m_1(\tau)-\v^m_2(\tau)\right\|_{\V}^2 \,\mathrm{d}\tau,
\end{equation}
for all $t \in[0, T]$, where $\theta_1^m$, $\theta_2^m$ are two solutions to problem (\ref{semi:thetamequation}) corresponding to velocities $\v_1^m$, $\v_2^m$.
\end{lemma}
\medskip

Since $\v^m$ is sufficiently regular, the existence of a global weak solution $\theta^m$ satisfying the maximum principle follows from that in \cite{Lorca}. Besides, we refer to \cite{HW2017} for the continuous dependence estimate \eqref{continuous1} and to \cite{ZZF2013} for the higher order estimate \eqref{semi:thetamenergy2} (via the transform $\Theta^m(x,t) = \int_0^{\theta^m(x,t)} \kappa(s)\, \mathrm{d}s$).
\medskip 

\textbf{Solvability of  $(\phi^m,\mu^m)$.}
Consider the system for $(\phi^m,\mu^m)$:
\begin{equation} \label{phim}
\begin{cases}
\phi_t^m  + \v^m \cdot \nabla \phi^m =\Delta \mu^m,
\quad \text { in }\Omega\times (0, T), \\
\mu^m =  - \Delta \phi^m + W^{\prime}(\phi^m),
\quad \text { in }\Omega\times (0, T), \\
\left.\partial_\mathbf{n} \phi^m\right|_{\partial \Omega}
= \left.  \partial_\mathbf{n} \mu^m \right|_{\partial \Omega}=0,\quad \text{on}\ \partial\Omega\times (0,T),\\
\left.\phi^m\right|_{t=0}=  \phi_0.
\end{cases}
\end{equation} 
We have
\medskip

\begin{lemma}\label{semi:phimlemma}
Assume that $\v^m \in C\left([0, T] ; \H_m\right)$ and $\phi_0 \in H^{1}(\Omega)$ with $\left\|\phi_0\right\|_{L^{\infty}} \leqslant 1$,
$\left|\overline{\phi_0}\right|<1$. Then problem (\ref{phim}) admits a unique weak solution
$$\phi^m \in L^{\infty}(0,T;H^{1}(\Omega)) \cap L^{4}(0,T;H^{2}(\Omega)) \cap L^2(0,T;W^{2,p}(\Omega))\cap H^{1}(0,T;(H^{1}(\Omega))^{\prime}),$$  for any $p\geqslant 2$, $\mu^m \in {L^2(0,T;H^1(\Omega))}$, with
$ \phi^m \in L^{\infty}(\Omega \times (0,T))$, $|\phi^m|<1$ almost everywhere in $\Omega \times (0,T)$ and $\sup \limits_{0\leqslant t\leqslant T} \| \phi^m(t) \|_{L^{\infty}} \leqslant 1$. Moreover, we have the continuous dependence estimate:
\begin{equation} \label{continuous2}
\left\|\phi_1^m(t)-\phi_2^m(t)\right\|_{V_0^{\prime}}^2 +
\int_0^t \| \nabla (\phi_1^m-\phi_2^m)(\tau) \|^2\, \mathrm{d}\tau  \leqslant
C \int_0^t \left\|\v^m_1(\tau)-\v^m_2(\tau)\right\|^2 \,\mathrm{d} \tau,
\end{equation}
for all $t\in [0,T]$, where $\phi_1^m$, $\phi_2^m$ are two solutions to problem (\ref{phim}) corresponding to velocities $\v_1^m$, $\v_2^m$.
\end{lemma} 
\medskip 

Since $\v^m$ is sufficiently regular, we refer to \cite[Theorem 6]{Abels2009} for the existence and uniqueness of weak solution, see also \cite{HwuHeleshaw} for the regularity $\phi^m\in L^{4}(0,T;H^{2}(\Omega))$. The continuous dependence estimate follows from the argument in \cite{HeWeak}, in which a more general system was studied. 

\subsection{Solvability of $\u^m$ with given $(\phi^m,\mu^m,\theta^m)$}
Thanks to Lemmas \ref{semi:thetamlemma}, \ref{semi:phimlemma}, we have $\theta^m \in L^{\infty}\left(0, T ; H_0^1 \cap L^{\infty}\right) \cap L^2\left(0, T ; H^2\right)$, $\phi^m \in L^{\infty}(0,T;H^{1}) \cap L^{4}(0,T;H^{2}) \cap H^{1}(0,T;(H^{1})^{\prime})$, such that $\sup \limits_{0\leqslant t\leqslant T} \left\|\theta^m(t)\right\|_{L^{\infty}} \leqslant\left\|\theta_0\right\|_{L^{\infty}}$ and 
 $\sup \limits_{0\leqslant t\leqslant T} \| \phi^m(t) \|_{L^{\infty}} \leqslant 1$. Given the pair $(\phi^m,\theta^m)$, we consider the system for $\u^m$:
\begin{equation} \label{um}
\begin{cases}
(\u_t^m, \w^m) + (\u^m \cdot \nabla \u^m, \w^m)+( 2 \nu(\theta^m) \D \u^m, \nabla \w^m) \\
=\int_{\Omega}[\lambda(\theta^m) \nabla \phi^m \otimes \nabla \phi^m]: \nabla \w^m \,\mathrm{d} x
+\int_{\Omega} \theta^m \g \cdot \w^m \,\mathrm{d}x,  \\
\left.\u^m\right|_{t=0} =\Pi_{m}\u_0,
\end{cases}
\end{equation}
for any $\w^m \in \H_m$.
The system (\ref{um}) is equivalent to a Cauchy problem of ordinary differential equations. Hence, for any $\u_0 \in \L(\Omega)$, the classical Cauchy-Lipschitz Theorem guarantees the existence and uniqueness of a local solution $\u^m \in H^1\left(0, \mathcal{T}_m ; \H_m\right)$.
Using the fact that $\u^m(\cdot, t)$ lies in a finite-dimensional space, we can apply the argument in \cite{HW2017} to obtain the following estimate:
\begin{align}
&\sup\limits_{t \in [0,\mathcal{T}_m]} \| \u^m(t) \|^2 + \int_0^{\mathcal{T}_m} \| \nabla \u^m (t) \|^2 \,\d t \notag \\
&\quad \leqslant C \left(
\| \u_0 \|^2 + \int_0^{\mathcal{T}_m} \| \nabla \phi^m(t) \|^4 \,\d \tau + \int_0^{\mathcal{T}_m} \| \theta^m (t) \|^2 \d\, \tau \right),
\label{boundcontrol2}
\end{align}
which enables us to extend the solution $\u^m$ to the interval $[0,T]$.

For completeness, we show the continuous dependence of $\u^m$ with respect to $(\phi^m,\theta^m)$. Let $\u^m_1$, $\u^m_2$ be two solutions with the same initial data  corresponding to $(\phi^m_1,\theta^m_1)$ and $(\phi^m_2,\theta^m_2)$, respectively.
Let
\begin{align}
\u^m = \u^m_1 - \u^m_2, \quad \phi^m = \phi^m_1 - \phi^m_2,\quad  \theta^m = \theta^m_1 - \theta^m_2.
\notag 
\end{align}
It follows that
\begin{align}
&\left( \u_t^m, \w^m\right) +
 \int_{\Omega}\left(\u_1^m \cdot \nabla \u^m+\u^m \cdot \nabla \u_2^m\right) \cdot \w^m \,\mathrm{d} x +
\int_{\Omega} 2 \nu\left(\theta_1^m \right) \D \u^m: \nabla \w^m \,\mathrm{d} x
\notag \\
&+\int_{\Omega} 2 \left(\nu\left(\theta_1^m\right)-\nu\left(\theta_2^m\right)\right) \D \u_2^m: \nabla \w^m \,\mathrm{d} x
\notag \\
&= \int_{\Omega}\left(\lambda\left(\theta_1^m\right) \nabla \phi_1^m \otimes \nabla \phi_1^m
-\lambda\left(\theta_2^m\right) \nabla \phi_2^m \otimes \nabla \phi_2^m\right): \nabla \w^m \,\mathrm{d} x
+ \int_{\Omega} \theta^m \g \cdot \w^m\, \mathrm{d} x,
\label{umuniqueness1}
\end{align}
for any $\w^m \in \H_m$. 
In view of the regularity of $\u_i^m$, we can take the test function $\w^m = \u^m$. Thanks to the Korn's inequality \eqref{Korn}, we deduce that
\begin{align*}
&\frac{1}{2} \frac{\d}{\dt} \| \u^m \|^2 + \int_{\Omega} \nu(\theta_1^m) | \nabla \u^m |^2 \,\dx \\
&\quad \leqslant - \int_{\Omega} (\u^m \cdot \nabla \u_2^m) \cdot \u^m \,\dx
- \int_{\Omega} 2 (\nu(\theta_1^m) - \nu(\theta_2^m)) \D \u_2^m : \nabla \u^m \,\dx \\
&\qquad + \int_{\Omega} (\lambda (\theta_1^m) \nabla \phi_1^m \otimes \nabla \phi_1^m
-\lambda(\theta_2^m) \nabla \phi_2^m \otimes \nabla \phi_2^m): \nabla \u^m \,\dx
+ \int_{\Omega} \theta^m \g \cdot \u^m \,\dx.
\end{align*}
Thanks to the finite-dimensional property of $\H_m$, we have for $i=1,2$, 
$$
\|\u^m_i\|_{\boldsymbol{L}^\infty}\leq C \|\u^m_i\|, \quad  \|\nabla \u^m_i\|_{\boldsymbol{L}^\infty}\leq C \|\nabla \u^m_i\|
\leq C\|\u^m_i\|,
$$
where the positive constant $C$ depends on $m$. Similar results also hold for $\u^m$, $(\u_i^m)_t$ and $\u_t^m$. Then we can deduce that
\begin{align}
- \int_{\Omega} (\u^m \cdot \nabla \u_2^m) \cdot \u^m \dx
&\leqslant \| \u^m \| \| \nabla \u_2^m \|_{\boldsymbol{L}^{\infty}} \| \u^m \|	 \leqslant C \| \u^m \|^2, \notag 
\end{align}
\begin{align}
- \int_{\Omega} 2 (\nu(\theta_1^m) - \nu(\theta_2^m)) \D \u_2^m : \nabla \u^m \dx
&\leqslant C \| \theta^m \| \| \nabla \u_2^m \|_{\boldsymbol{L}^{\infty}} \| \nabla \u^m \| \notag \\
&\leqslant \varepsilon \| \nabla \u^m \|^2 + C \| \theta^m \|^2,\notag 
\end{align}
and 
\begin{align}
\int_{\Omega} \theta^m \g \cdot \u^m \dx \leqslant C \| \theta^m \|^2 + C \| \u^m \|^2.\notag 
\end{align}
Next, by the Gagliardo-Nirenberg inequality and  Lemmas \ref{semi:thetamlemma}, \ref{semi:phimlemma}, we obtain
\begin{align}
&\int_{\Omega} (\lambda (\theta_1^m) \nabla \phi_1^m \otimes \nabla \phi_1^m
-\lambda(\theta_2^m) \nabla \phi_2^m \otimes \nabla \phi_2^m) : \nabla \u^m \,\dx \notag \\
&\quad =\int_{\Omega} \big[  (\lambda(\theta_1^m) - \lambda(\theta_2^m) ) \nabla \phi_2^m \otimes \nabla \phi_2^m
+ \lambda(\theta_1^m) (\nabla \phi^m \otimes \nabla \phi_2^m + \nabla \phi_1^m \otimes \nabla \phi_1^m) \big]: \nabla \u^m \,\dx \notag \\
&\quad \leqslant C \| \theta^m \| \| \nabla \phi_2^m \|_{L^4}^2 \| \nabla \u^m \|_{\boldsymbol{L}^{\infty}}
+ C(\| \nabla \phi_1^m \| + \| \nabla \phi_2^m \|) \| \nabla \phi^m \| \| \nabla \u^m \| _{\boldsymbol{L}^{\infty}} \notag \\
&\quad \leqslant C \| \theta^m \|^2 + C \| \phi_2^m \|_{H^2}^2 \| \u^m \|^2
+ C \| \nabla \phi^m \|^2 + C \| \u^m \|^2.
\notag 
\end{align}
Combining the above estimates and taking $\varepsilon$ sufficiently small, we have
\begin{align} \label{umuniqueness2}
\frac{1}{2} \frac{\d}{\dt} \|\u^m\|^2 +  \frac{\underline{\nu}}{2} \| \nabla \u^m\|^2
\leqslant
C (1 + \| \phi_2^m \|_{H^2}^2) \|\u^m\|^2 +  C ( \|\nabla \phi^m\|^2 + \| \theta^m\|^2  ),
\end{align}
which together with Gronwall's lemma implies
\begin{equation} \label{continuous3}
	  \left\|\u^m(t)\right\|^2 + \int_0^t  \| \nabla \u^m(\tau)\|^2 \,\mathrm{d} \tau  \leqslant C  \left(\int_0^t \| \nabla \phi^m (\tau)\|^2 + \| \theta^m(\tau) \|^2 \,\mathrm{d} \tau \right) e^{C t},
\end{equation}
for all $t\in [0,T]$. 
Next, taking $\w^m= \u_t^m$ in \eqref{umuniqueness1}, we get 
\begin{align*}
\| \u_t^m\|^2
&= -
 \int_{\Omega}\left(\u_1^m \cdot \nabla \u^m+\u^m \cdot \nabla \u_2^m\right) \cdot \u_t^m\,\mathrm{d} x 
 - \int_{\Omega} 2 \nu\left(\theta_1^m \right) \D \u^m: \nabla \u_t^m \,\mathrm{d} x
\notag \\
&\quad - \int_{\Omega} 2 \left(\nu\left(\theta_1^m\right)-\nu\left(\theta_2^m\right)\right) \D \u_2^m: \nabla \u_t^m \,\mathrm{d} x + \int_{\Omega} \theta^m \g \cdot \u_t^m\, \mathrm{d} x
\notag \\
&\quad + \int_{\Omega}\left(\lambda\left(\theta_1^m\right) \nabla \phi_1^m \otimes \nabla \phi_1^m
-\lambda\left(\theta_2^m\right) \nabla \phi_2^m \otimes \nabla \phi_2^m\right): \nabla\u_t^m \,\mathrm{d} x\\
&\leqslant \|\u_1^m\|_{\boldsymbol{L}^\infty}\| \nabla \u^m\| \|\u_t^m\|
+ \|\u^m\|\|\nabla \u_2^m\|_{\boldsymbol{L}^\infty} \|\u_t^m\| \\
&\quad +C \|\nabla \u^m\| \|\nabla \u_t^m\|
+C \| \theta^m \| \| \nabla \u_2^m \|_{\boldsymbol{L}^{\infty}} \| \nabla \u_t^m \| + \|\theta^m\|\| \u_t^m\|\\
&\quad + C \| \theta^m \| \| \nabla \phi_2^m \|_{L^4}^2 \| \nabla \u_t^m \|_{\boldsymbol{L}^{\infty}}
+ C(\| \nabla \phi_1^m \| + \| \nabla \phi_2^m \|) \| \nabla \phi^m \| \| \nabla \u_t^m \| _{\boldsymbol{L}^{\infty}}\\
&\leqslant \frac12 \|\u_t^m \|^2 + C\| \nabla \u^m\|^2+ C\| \u^m\|^2 +  C(1+\| \phi_2^m \|_{H^2}^2) \| \theta^m \|^2 + C\|\nabla \phi^m\|^2.
\end{align*}
Integrating over time, we deduce from \eqref{continuous3} that
\begin{equation} \label{continuous4}
\int_0^t \| \u_t^m(\tau)\|^2\, \mathrm{d} \tau  \leqslant C \left(
\int_0^t \| \nabla \phi^m(\tau) \|^2\, \mathrm{d} \tau  +
\sup\limits_{\tau \in [0,t]}  \| \theta^m(\tau)\|^2 \right),\quad \forall\, t\in [0,T].
\end{equation}

\subsection{The fixed point argument}
Combining the continuous dependence estimates (\ref{continuous1}), (\ref{continuous2}), (\ref{continuous3}), (\ref{continuous4}), we have
\begin{align} 
& \sup \limits_{t\in [0,T]} \|\u_1^m(t)-\u_2^m(t)\|^2 + \int_0^T  \| \nabla (\u_1^m(t)-\u_2^m(t))\|^2 \,\mathrm{d} t + \int_0^T \|  \partial_t(\u_1^m(t)-\u_2^m(t)) \|^2 \,\mathrm{d} t \notag
\\
&\quad \leqslant C_T \int_0^T \left\|\nabla (\v_1^m(t)-\v_2^m(t)) \right\|^2\, \mathrm{d} t. \label{lianxu}
\end{align}
Let
\begin{align*}
\boldsymbol{X} =& L^{\infty}\left(0, T ; H_0^1(\Omega) \cap L^{\infty}(\Omega)\right) \cap L^2\left(0, T ; H^2(\Omega)\right) \\
&\times L^{\infty}(0,T;H^{1}) \cap L^{4}(0,T;H^{2}) \cap H^{1}(0,T;(H^{1})^{\prime}).
\end{align*}
We infer from \eqref{lianxu} that the mapping
$$
\begin{array}{rlccc}
	& &  & &\\
	\Phi_T^m:\ C\left([0, T] ; \H_m\right) & \rightarrow &
	\textbf{X}
	& \rightarrow & H^1\left(0, T ; \H_m\right) \\
	\v^m & \mapsto & \left(\theta^m,\phi^m \right) & \mapsto & \u^m
\end{array}
$$
is continuous from $C\left([0, T] ; \H_m\right)$ to $H^1\left(0, T ; \H_m\right)$ and thus compact from $C\left([0, T] ; \H_m\right)$ into itself. 
Let $\left\|\v^m\right\|^2_{C([0, T] ; \H_m)}
=\sup\limits_{t \in[0,T]} \sum\limits_{i=1}^m\left|g_i^m(t)\right|^2 \leqslant M$.
Taking $M>0$ sufficiently large and $T_m$ sufficiently small, in view of (\ref{boundcontrol2}), we have
$\left\|\u^m\right\|^2_{C\left([0, T_m] ; \H_m\right)} \leqslant M$. 
Set 
$$
\boldsymbol{Y}_M=\big\{ \u\in C\left([0, T_m] ; \H_m\right)\ :\   \left\|\u\right\|^2_{C([0, T_m] ; \H_m)}  \leqslant M \big\},
$$ 
which is a bounded closed convex subset of $C\left([0, T_m] ; \H_m\right)$.
Therefore, the restriction of $\Phi_T^m$ on $\boldsymbol{Y}_M$ is compact and maps $\boldsymbol{Y}_M$ into itself. Hence, we can apply Schauder's fixed point theorem to obtain an approximate solution $(\u^m,\phi^m,\mu^m,\theta^m)$ of the semi-Galerkin scheme (\ref{semiu})-(\ref{semiboundary}) on $[0,T_m]$ with mentioned regularity. 

The proof of Theorem \ref{exe-approx} is complete.

\section{Semi-Galerkin Scheme for Strong Solutions} \label{semigalerkin_strongsolution}
The semi-Galerkin scheme (\ref{semiu})-(\ref{semiboundary}) is also valid for strong solutions (see \cite{HeStrong,GiorginiGalerkin} for a similar strategy). Compared with the regularity for $\phi^m, \mu^m, \theta^m$ obtained in the semi-Galerkin scheme for weak solutions, further higher order estimates of the approximate solutions require higher regularity for these variables.

For any positive integer $m$, we consider a given vector $\v^m \in C([0,T];\H_m)$. Concerning the subsystem (\ref{semi:thetamequation}) for $\theta^m$, Lemma \ref{semi:thetamlemma} is still valid. Moreover, under the additional assumption $\theta_0 \in H^2(\Omega)$, we can verify that the solution $\theta^m$ belongs to the higher order space $L^{\infty}(0,T;H^2) \cap L^2(0,T;H^3) \cap H^1(0,T;H_0^1)$ (see e.g. \cite{HW2017}). This can be achieved by applying the transform $\Theta^m(x,t) = \int_0^{\theta^m(x,t)} \kappa(s)\, \mathrm{d}s$ to eliminate the temperature-dependent thermal coefficient, and performing estimates for $\Theta^m$ as in \cite{ZZF2013}. Then we can recover the estimates for $\theta^m$. Next, we consider the convective Cahn-Hilliard system (\ref{phim}) for $(\phi^m,\mu^m)$. Lemma \ref{semi:phimlemma} is still valid. Since  $\phi_0 \in H^2(\Omega)$, $\mu_0 \in H^1(\Omega)$, higher order energy estimates can be obtained, following the arguments in \cite{Abels2009,Giorgini2019,HeStrong}. In particular, we have $\mu^m \in L^{\infty}(0,T;H^1)$. According to Remark \ref{remark mu and phi}, there exists $\delta = \delta_m \in (0,1)$, such that $\|\phi^m(t)\|_{L^{\infty}} \leqslant 1 - \delta_m$ for $t\in [0,T]$. Moreover, we can deduce that
\begin{align*}
	&\phi^m \in L^{\infty}(0,T;H^3) \cap L^2(0,T;H^4) \cap H^1(0,T;H^1),
    \\
    &\mu^m \in L^2(0,T;H^3)\cap H^1(0,T;(H^1)^{\prime}).
\end{align*}
 For further details, we refer to, for instance, \cite{Giorgini2019, HeStrong,Giorgini Heleshaw}. Given $(\phi^m,\theta^m)$ obtained above, we turn to the system (\ref{um}) for $\u^m$. This part is identical to the weak solution part. In particular, the continuous dependence estimates (\ref{continuous3}) and (\ref{continuous4}) hold. Finally, by the same fixed point argument, we can obtain a strong solution $(\u^m,\phi^m,\mu^m,\theta^m)$ of the semi-Galerkin scheme (\ref{semiu})-(\ref{semiboundary}) defined on some interval $[0,T_m]$ with higher order regularities.
 
 The proof of Theorem \ref{strongsemi} is complete.

\section{An Elliptic Problem with Singular Nonlinearity} \label{relation_of_mu_phi}
We consider the following Neumann problem
\begin{equation} \label{semielliptic}
	\begin{cases}
		- \Delta \phi + F^{\prime}(\phi) = \widetilde{\mu}, \quad \text{in}\ \Omega, \\
		 \partial_\mathbf{n} \phi  = 0,\quad \text{on}\ \partial\Omega.
	\end{cases}
\end{equation}
Here, $\widetilde{\mu} = \mu + B\phi$ (see (\ref{Wphi})). 
Then we have the following results (see e.g., \cite{Giorgini2019,GiorginiOnoo,HeStrong,inequality}):
\medskip
\begin{lemma} \label{mu and phi}
	Let $\Omega\subset \mathbb{R}^2$ be a bounded domain with smooth boundary. Assume that $\mu \in L^2(\Omega)$, then problem (\ref{semielliptic}) admits a unique solution $\phi\in H^2(\Omega)$ such that $ F^{\prime}(\phi) \in L^2(\Omega)$. Moreover,
	
	 (i) It holds $\tilde{\mu} \in L^2(\Omega)$ and
	\begin{align} \label{mu and phi1}
		\|\phi\|_{H^2(\Omega)}+\left\|F^{\prime}(\phi)\right\| \leqslant C(1+\|\widetilde{\mu}\|).
	\end{align}
	
	 (ii) If $\mu \in L^p(\Omega)$, where $2 \leqslant p \leqslant \infty$, then $\widetilde{\mu} \in L^p(\Omega)$, and we have
	\begin{align} \label{mu and phi2}
		\left\|F^{\prime}(\phi)\right\|_{L^p} \leqslant\|\widetilde{\mu}\|_{L^p}.
	\end{align}
	
	(iii) If $\mu \in H^1(\Omega)$, then $\widetilde{\mu} \in H^1(\Omega)$, and we have
	\begin{align} \label{mu and phi3}
		\|\Delta \phi\| \leqslant\|\nabla \phi\|^{\frac{1}{2}}\|\nabla \widetilde{\mu}  \|^{\frac{1}{2}}.
	\end{align}

(iv) Assume $\mu \in H^1(\Omega)$. For any $p\in [2,\infty)$, there exists a positive constant $C$ such that
	\begin{align} 
		& \|\phi\|_{W^{2, p}}+\left\|F^{\prime}(\phi)\right\|_{L^p} \leqslant C\left(1+\|\widetilde{\mu}\|_{H^1}\right), 
\label{mu and phi4} \\
        &		\left\|F^{\prime \prime}(\phi)\right\|_{L^p} \leqslant C\left(1+e^{C\|\widetilde{\mu}\|_{H^1}^2}\right).
        \label{mu and phi5}
	\end{align}
\end{lemma}
\begin{remark} \label{remark mu and phi}
	Lemma \ref{mu and phi} provides critical insights into the so-called strict separation property of $\phi$. Namely, the supreme of $\phi$ is controlled by the supreme of $F^{\prime}(\phi)$, while the latter can be controlled by the $H^1$-norm of the chemical potential $\mu$. Suppose $\mu \in H^1(\Omega)$, it follows from  \eqref{mu and phi4} and \eqref{mu and phi5} that $F^{\prime}(\phi) \in W^{1,p}(\Omega)$ for arbitrary $p \geq 2$ in two dimensions (see e.g., \cite[Lemma 3.2]{HeStrong}). This combined with the Sobolev embedding theorem entails $F^{\prime}(\phi) \in L^{\infty}(\Omega)$, and thus $\phi$ is strictly separated from the pure states $\pm1$ due the singular nature of $F$. For further discussions on the separation property of the phase variable $\phi$ satisfying the Cahn-Hilliard equation, we refer to \cite{GGM23,GiorginiOnoo,HeStrong,inequality} and the references therein.
\end{remark}

\section{Proof of Lemma \ref{Giorginiinequality}}
Below we prove the interpolation inequality (\ref{GiorginiHolder}). This relies on an embedding from the H\"{o}lder spaces into the Sobolev spaces with fractional order \cite{GiorginiHolder}.
\begin{proof}(\textbf{Proof of Lemma \ref{Giorginiinequality}.})
Let $u \in H^2(\Omega) \cap C^{\gamma}(\overline{\Omega})$, $\gamma\in (0,1)$. First, applying the Gagliardo-Nirenberg interpolation inequality for the fractional-order Sobolev spaces (see e.g., \cite[Theorem 1]{GLfraction}), we have
\begin{align*}
\| u \|_{W^{1,4}} \leqslant C \| u \|_{W^{s,p}}^{\xi} \| u \|_{H^2}^{1-\xi},
\end{align*}	
where
$$1=s\xi  +2(1-\xi),\qquad \frac{1}{4}=\frac{\xi}{p}+\frac{1-\xi}{2}, \qquad \xi\in (0,1).$$
Let us require $s= 2-\frac{1}{\xi} \in (0,\gamma)\subset (0,1)$. Then it follows that
$$\xi\in \left(\frac12,\frac{1}{2-\gamma}\right)\subset \left(\frac12,1\right),\qquad
p=\frac{4\xi}{2\xi-1}>\frac{4}{\gamma}.$$
Next, we show that $C^{\gamma} (\overline{\Omega}) \hookrightarrow W^{s,p}(\Omega)$ when $s\in (0,\gamma)$. To this end, we only need to estimate the Gagliardo seminorm (see e.g., \cite{GLfraction,fractional}). Indeed, the condition $s-\gamma < 0$ implies $n + (s-\gamma)p < n$ so that 
\begin{align*}
[u]_{W^{s,p}}^p
&= \int_{\Omega}\int_{\Omega} \frac{|u(x) - u(y)|^p}{|x-y|^{n+sp}} \, \d x \d y   \\
&\leqslant \int_{\Omega}\int_{\Omega} \frac{ \|u\|_{C^{\gamma}(\overline{\Omega})}^p  |x - y|^{\gamma p}}{|x-y|^{n+sp}} \, \d x \d y   \\
&\leqslant \|u\|_{C^{\gamma}(\overline{\Omega})}^p \int_{\Omega}\int_{\Omega} \frac{1}{|x-y|^{n+(s-\gamma)p}} \d x \d y \\
&\leqslant C \|u\|_{C^{\gamma}(\overline{\Omega})}^p.
\end{align*}
The proof is complete.
\end{proof}

%%%%%%%%%%%%%%%%%%%%%%%%%%%%%%%%%%%%%%%%%%%%%%%%%

\end{document}